\documentclass{amsart}

\usepackage{amscd,amsmath,amssymb,amsfonts,bbm, calligra, xspace}
\usepackage[T1]{fontenc}
\usepackage{lmodern}
\usepackage{mathtools}
\usepackage{enumerate, enumitem}
\usepackage{multicol}
\usepackage[all]{xy}
\usepackage{mathrsfs}
\usepackage{footmisc}

\usepackage{xcolor}
\definecolor{darkblue}{rgb}{0,0,0.8}
\definecolor{darkgreen}{rgb}{0,0.4,0}
\usepackage[colorlinks=true,linkcolor=darkblue, citecolor=darkgreen,  urlcolor=darkgreen, unicode,pdfborder={0 0 0},final]{hyperref}

\newtheorem{thm}{Theorem}[section]
\newtheorem{prop}[thm]{Proposition}

\newtheorem{lem}[thm]{Lemma}
\newtheorem{cor}[thm]{Corollary}

\theoremstyle{definition}

\theoremstyle{remark}
\newtheorem{rem}[thm]{Remark}
\newtheorem{rems}[thm]{Remarks}

\newtheorem{Step}{Step}

\numberwithin{equation}{section}

\newcommand{\Ker}{\mathrm{Ker}}

\newcommand{\cl}{\mathrm{cl}}

\newcommand{\Zar}{\mathrm{Zar}}

\newcommand{\et}{\mathrm{\acute{e}t}}
\newcommand{\red}{\mathrm{red}}

\newcommand{\an}{\mathrm{an}}
\newcommand{\ab}{\mathrm{ab}}

\newcommand{\Frob}{\mathrm{Frob}}
\newcommand{\Spec}{\mathrm{Spec}}

\newcommand{\Frac}{\mathrm{Frac}}

\newcommand{\Gal}{\mathrm{Gal}}

\newcommand{\Div}{\mathrm{Div}}
\newcommand{\Pic}{\mathrm{Pic}}
\newcommand{\Br}{\mathrm{Br}}
\newcommand{\ddiv}{\mathrm{div}}

\newcommand{\h}{\mathrm{h}}

\newcommand{\llangle}{\langle\kern-.1em\langle}
\newcommand{\rrangle}{\rangle\kern-.1em\rangle}

\newcommand{\isoto}{\myxrightarrow{\,\sim\,}}
\makeatletter
\def\myrightarrow{{\setbox\z@\hbox{$\rightarrow$}\dimen0\ht\z@\multiply\dimen0 6\divide\dimen0 10\ht\z@\dimen0\box\z@}}
\def\myrightarrowfill@{\arrowfill@\relbar\relbar\myrightarrow}
\newcommand{\myxrightarrow}[2][]{\ext@arrow 0359\myrightarrowfill@{#1}{#2}}
\makeatother

\def\loccit{\emph{loc}.\kern3pt \emph{cit}.{}\xspace}
\def\eg{e.g.\kern.3em}
\def\ie{i.e.,\ }
\def\resp {\text{resp.}\kern.3em}

\def\Z{\mathbb Z}

\def\F{\mathbb F}

\def\Q{\mathbb Q}
\def\P{\mathbb P}
\def\R{\mathbb R}

\def\cA{\mathcal{A}}

\def\cN{\mathcal{N}}
\def\cM{\mathcal{M}}
\def\cL{\mathcal{L}}
\def\cO{\mathcal{O}}

\def\cJ{\mathcal{J}}

\def\cM{\mathcal{M}}
\def\cP{\mathcal{P}}

\def\cI{\mathcal{I}}

\def\km{\mathfrak{m}}

\def\ka{\mathfrak{a}}

\def\tq{\tilde{q}}

\def\oF{\overline{\mathbb{F}}}

\def\wP{\widetilde{\mathbb{P}}}

\def\wD{\widetilde{D}}
\def\wE{\widetilde{E}}

\def\wX{\widetilde{X}}

\def\ws{\widetilde{s}}

\def\whX{\widehat{X}}

\begin{document}

\title[The Pythagoras number of fields of transcendence degree \texorpdfstring{$1$}{1} over \texorpdfstring{$\Q$}{Q}]{The Pythagoras number of fields of transcendence degree \texorpdfstring{$1$}{1} over \texorpdfstring{$\Q$}{Q}}

\author{Olivier Benoist}
\address{D\'epartement de math\'ematiques et applications, \'Ecole normale sup\'erieure, CNRS,
45 rue d'Ulm, 75230 Paris Cedex 05, France}
\email{olivier.benoist@ens.fr}

\renewcommand{\abstractname}{Abstract}

\begin{abstract}
We show that any sum of squares in a field of transcendence degree~$1$ over $\Q$ is a sum of $5$ squares,
answering a question of Pop and Pfister. 
We deduce this result from a representation theorem, in $k(C)$,  for qua\-dratic forms of rank $\geq 5$ with coefficients in $k$, where $C$ is a curve over a number~field~$k$.
\end{abstract}

\maketitle

\section{Introduction}

\subsection{Sums of squares and Pythagoras numbers}

The \textit{Pythagoras number} $p(F)$ of a field $F$ is the smallest integer~$p$ such that any sum of squares in $F$ is a sum of $p$ squares in $F$ (or $+\infty$ if no such integer exists). It measures the complexity of sums of squares problems in the field $F$. We refer to \cite[Chap.\,7 \S 1]{Pfisterbook} for a general introduction to this invariant.

Euler proved in \cite{Euler} that a nonnegative rational number is a sum of $4$ squares of rational numbers (see \cite{Pieper} for historical comments). Siegel \cite{Siegel} extended Euler's theorem to an arbitrary number field: an element of a number field~$k$ that is nonnegative in all its real embeddings is a sum of $4$ squares in $k$. These results imply that $p(k)\leq 4$
(and in fact $p(\Q)=4$ as~$7$ is not a sum of $3$ squares in~$\Q$). 

In \cite{Pourchet}, Pourchet considered the case of the field $k(t)$ of rational functions in one variable over a number field $k$. There, he proved that~$p(k(t))\leq 5$ for any number field $k$ (and in fact~$p(\Q(t))=5$ as~$t^2+7$ is not a sum of $4$ squares in~$\Q(t)$), thereby improving on an earlier bound $p(\Q(t))\leq 8$ due to Landau \cite{Landau}. 
 %[Hsia-Johnson, On the representation...] completely computes p(k(t)), but really this is nothing more than Pourchet.
 %There is an exposition in [Rajwade, Squares] but this is identical to Pourchet.

\vspace{1em}

The aim of the present article is to prove a similar result for general fields $F$ of transcendence degree $1$ over $\Q$. For such fields, the inequality $p(F)\leq 7$ was shown by Colliot-Th\'el\`ene in \cite{CTAppendix}. The stronger bound $p(F)\leq 6$ was obtained by Pop in~1991~\cite{Popnote} (his result was published only recently in \cite{Pop}). Whether this is optimal or not remained open (see \cite[Chap.\,7 Conjecture~1.10\,(1)]{Pfisterbook} and \mbox{\cite[\S 3]{Pop}}). We prove that it is not optimal, and that Pourchet's bound on the Pythagoras number of $k(t)$ can be extended to all these fields.

\begin{thm}
\label{main1}
Let $F$ be a field of transcendence degree $1$ over $\Q$. Then $p(F)\leq 5$.
\end{thm}

We refer to Corollary \ref{17} for a reformulation of Theorem~\ref{main1} that clarifies its relation with Hilbert's 17th problem.

If $F$ is a field of transcendence degree $d\geq 2$ over $\Q$, then the bound ${p(F)\leq 2^{d+1}}$ was obtained by Colliot-Th\'el\`ene and Jannsen (see \cite{CTAppendix,CTJ, Jannsen}) as a consequence of the Milnor conjectures proven by Voevodsky in \cite{Voevodsky} and of a local-global principle conjectured by Kato in \cite{Kato} and proven by Jannsen in~\cite{Jannsen}. This bound is not known to be optimal (when $d=2$, this problem is closely related to the questions raised by Jannsen and Sujatha in \cite[Remarks~8]{JS}).

\subsection{Quadratic forms over function fields of curves over number fields}
\label{parth2}

 The Hasse--Minkowski theorem, proved by Minkowski over $\Q$ and by Hasse over a general number field, implies the following far-reaching generalization of the theorems of Euler and Siegel (\cite[Satz 19]{Hasse}, see also \cite[I.3.5 and VI.3.5]{Lam}).
Let~$q$ be a nondegenerate quadratic form of rank $\geq 4$ over a number field $k$. Then $q$ represents an element of~$k$ if and only if this element is nonnegative (\resp nonpositive) at the real places of $k$ at which $q$ is positive definite (\resp negative definite).

In \cite{Pourchet}, Pourchet carried out  a similar analysis over the field $k(t)$ of rational functions in one variable over a number field $k$. There, he gives necessary and sufficient conditions for an element of $k(t)$ to be represented by a given nondegenerate quadratic form of rank $\geq 5$ over $k$ (see \cite[Corollaire 1 p.\,98]
%, Remarque p.\,96
{Pourchet} or Corollary \ref{coroPourchet}).

\vspace{1em}

Let $C$ be a geometrically connected smooth projective curve over a number~field~$k$. Let $q$ be a nondegenerate quadratic form of rank $\geq 5$ over $k$. Following Pourchet's lead, we derive Theorem \ref{main1} from a
general representation theorem for~$q$ over~$k(C)$, which is our main technical result. To state it, we introduce some~notation.

Let $k_v$ be the completion of $k$ with respect to a place $v$. We let $q_v$ and~$C_v$ be~the quadratic form and the curve over $k_v$ obtained from $q$ and $C$ by extension of scalars.

Let $v$ be a real place of $k$. A rational function~$f\in k(C)$ is said to be \textit{nonnegative} (\resp \textit{nonpositive}) at $v$ if $f(x)\geq 0$ (\resp ${f(x)\leq 0}$) for all~${x\in C_v(\R)}$ that is not a pole of $f$. If $D$ is a divisor on $C_v$, we let $\cl_v(D)\in H^1(C_v(\R),\Z/2)=(\Z/2)^{\pi_0(C_v(\R))}$ be the element associating with any connected component~$\Gamma$ of $C_v(\R)$ the parity of the number of points of~$D$ lying on $\Gamma$ (counted with multiplicity). The morphism $\cl_v:\Div(C_v)\to H^1(C_v(\R),\Z/2)$ factors through rational equivalence (as a rational function changes sign an even number of times on each~$\Gamma$) and induces a morphism $\cl_v:\Pic(C_v)\to H^1(C_v(\R),\Z/2)$ called the \textit{Borel--Haefliger cycle class map} of $C_v$.

\begin{thm}
\label{main2}
Let $C$ be a geometrically connected smooth projective curve over a number field~$k$. Let $q$ be a nondegenerate quadratic form of rank $r\geq 5$ over $k$. Fix~$f\in k(C)^*$. Write $\ddiv(f)=E-2D$ with $E$ a reduced effective divisor. 

Then~$q$ represents $f$ in $k(C)$ if and only if there exists $\cM\in\Pic(C)$
such that:
\begin{enumerate}[label=(\roman*)] 
\item
\label{intro}
 if $v$ is a real place of $k$ and $q_v$ is positive definite (\resp negative definite), then $f$ is nonnegative  (\resp nonpositive) at $v$ and $\cl_v(\cM)=0$;
 \item
 \label{iintro}
if $v$ is a place of $k$ with $q_v=\tilde{q}_v\perp\langle 1,-1\rangle$ for some quadratic form~$\tq_v$ over~$k_v$, then
there exist a line bundle $\cP\in\Pic(C_v)$ and a divisor $\Delta$ on~$C_v$ with ${\cM\otimes\cP^{\otimes 2}\simeq\cO_{C_v}(\Delta-D)}$, such that~$f$ is invertible at $x$ and $f(x)$ is represented by $\tilde{q}_v$ for all closed points $x$ in the support of~$\Delta$.
\end{enumerate}
\end{thm}

When $C=\P^1_k$, Theorem \ref{main2} can be used to recover Pourchet's results \cite[Corollaire~1~p.\,98]{Pourchet} (see Corollary \ref{coroPourchet}).
The main novelty of Theorem \ref{main2}, however, 
is that it applies to any curve~$C$, possibly with more complicated geometry than $\P^1_k$ (with higher Picard rank, disconnected real loci, or finite places of bad reduction). 

Being formulated in terms of the existence of a line bundle $\cM$ on~$C$, the representation criterion given by Theorem \ref{main2} is influenced by arithmetic properties of the curve $C$ (\eg by the Mordell--Weil group of it Jacobian).

Additional work leads to significant simplifications to the statement of Theorem~\ref{main2} in many particular instances.
 This includes the cases when $C$ is a nonsplit conic over $k$ (see Corollary \ref{coros}\,\ref{dtext}), or when $q$ does not have hyperbolic signature at any real place of $k$ (see Corollary \ref{coros}\,\ref{ctext}; this is the consequence of Theorem \ref{main2} that leads to a proof of Theorem \ref{main1}).

\subsection{Analyzing and applying Theorem \ref{main2}}

Together, conditions \ref{intro} and \ref{iintro} in Theorem \ref{main2} constrain $\cM$ at all places of $k$.
Indeed, unless $v$ is real and $q_v$ is definite, the form~$q_v$ is isotropic (see \cite[VI.2.12]{Lam}).

However, these constraints are nontrivial only at finitely many places of $k$. To see it, note that condition~\ref{iintro} is not vacuous only if $\tq_v$ is anisotropic (as isotropic quadratic forms are universal).
%see \cite[I.3.4]{Lam}
This is the case exactly when $v$ is a real place at which $q_v$ has hyperbolic signature $(r-1,1)$ or $(1,r-1)$, or if $r=5$ and $v$ is a $p$-adic place at which the Hasse invariant of $q_v$ is nontrivial (use \cite[VI.2.12]{Lam}).

In particular, if $r\geq 6$, there are no constraints whatsoever at the $p$-adic places of $k$. In retrospect, this explains why the bound $p(k(C))\leq 6$ due to Pop is easier to prove than the bound $p(k(C))\leq 5$ provided by Theorem \ref{main1}.

\vspace{1em}
 
Let us further comment on condition \ref{iintro}.
If $v$ is a real place of $k$ and $q_v$ has signature~$(r-1,1)$ (\resp $(1,r-1)$), this condition is equivalent to the following assertion: for any connected component $\Gamma$ of $C_v(\R)$ on which~$f$ is nonpositive (\resp nonnegative), the class $\cl_v(\cM(D))|_{\Gamma}\in H^1(\Gamma,\Z/2)=\Z/2$ vanishes (see Remark \ref{rem41} (ii)). It is therefore not difficult to check in practice.

Condition \ref{iintro} is harder to analyze in general at a $p$-adic place $v$ of $k$. 
In this case, we prove that condition~\ref{iintro} is always satisfied for~${\cM=\cO_{C}}$ (see Proposition~\ref{propmiracle}). 
This surprising and nontrivial fact is key to the proofs of Corollary~\ref{coros}\,\ref{ctext} and, through it, of Theorem \ref{main1}.
Its proof is based on the geometric study of a projective regular model $\pi:Y\to\Spec(\cO_{k_v})$ of~$C_v$, on the Lang--Weil estimates, and on global class field theory applied to the function fields of the irreducible components of the special fiber of~$Y$.
There is no counterpart to these arguments in Pourchet's proof of the inequality $p(k(t))\leq 5$.

\subsection{Proving Theorem \ref{main2}}
\label{parsketch}

Our proof of Theorem \ref{main2} is deeply inspired by Pourchet's work.
We now outline the argument, assuming for simplicity that $q$ is the sum~of~$5$ squares quadratic form $\langle 1,1,1,1,1\rangle$ (the one that is used in the application to Theorem~\ref{main1}).

\vspace{1em}

A crucial input are two theorems of Kato (the local-global principle \cite[Theorem 0.8\,(2)]{Kato} and the local criterion \cite[Proposition~5.2]{Kato}). Combined with the Merkurjev--Suslin theorem, they control sums of $4$ squares in $k(C)$ through the following two statements. First, an element of $k(C)$ is a sum of~$4$ squares if and only if it is a sum of~$4$ squares in $k_v(C_v)$ for all places $v$ of~$k$ (see Proposition \ref{localglobal}). Second, if $v$ is a $p$-adic place of $k$ for some prime number $p$, an element of~$k_v(C_v)$ is a sum of~$4$ squares if and only if it is a sum of~$4$ squares in the henselizations of~$k_v(C_v)$ with respect to all integral divisors (both horizontal and vertical) of a proper regular model of $C_v$ over $\Spec(\cO_{k_v})$ (see Proposition \ref{local}).
% If $p$ is odd, the latter statement also follows from
%the more general \cite[Theorem 3.1, Remarks 3.2 and 3.6]{CTPS}.
%As the remarks indicate, in CTPS, need to choose the model appropriately.

Kato's local-global principle was applied in this precise way in the above-men\-tioned works \cite{CTAppendix, CTJ, Popnote, Jannsen}, but our use of Kato's local criterion is new. When $C=\P^1_k$, these results are much easier to prove (for instance, one may disregard the vertical divisors in the local criterion), and were already used by Pourchet (see \cite[Propositions 3 et 4]{Pourchet}).

\vspace{1em}

 Let $f\in k(C)^*$ be nonnegative at all real places of $k$. We wish to show that~$f$ is a sum of $5$ squares. Write~$f=\frac{\sigma}{\tau^2}$, where~$\sigma\in H^0(C,\cL^{\otimes 2})$ has reduced zero locus, and~$\tau$ is a rational section of $\cL$, for some $\cL\in \Pic(C)$. (In the notation of Theorem~\ref{main2}, one has~$D=\{\tau=0\}$ and~$E=\{\sigma=0\}$.) We consider the equation 
\begin{equation}
\label{equalb}
\sigma\alpha^2=\sum_{i=1}^5\beta_i^2,
\end{equation}
where $\alpha\in H^0(C,\cM)$ and $\beta_i\in H^0(C,\cL\otimes \cM)$, for some
line bundle $\cM$ on~$C$ (also requiring that $\alpha\neq 0$ and that the $\beta_i$ have no common zero).
Manipulating sections of line bundles instead of rational functions is a new and essential feature of our approach. Shadows of it appear in Pourchet's work in the guise of conditions on the degree of polynomials (see \eg conditions (2) in \cite[Th\'eor\`emes 1 et 2]{Pourchet}).

\vspace{1em}

We first solve (\ref{equalb}) locally. If $\cM$ is sufficiently ample, it turns out that (\ref{equalb}) has a local solution at a place~$v$ of $k$ if and only if \ref{intro} is satisfied (if $v$ is real and~$q_v$ is definite), or~\ref{iintro} is satisfied (otherwise). In the first case, the assertion is a consequence of theorems of Witt \cite{Witt} (see Proposition~\ref{condition1prop}).
The second case hinges on a purely geometric analysis of the birational geometry of the quadric bundle $\{\sigma T_0^2=\sum_{i=1}^5 T_i^2\}$ over~$C_v$ (see Proposition~\ref{condition2prop}).

Suppose now that local solutions $(\alpha_v,(\beta_{i,v})_{1\leq i\leq 5})$ of (\ref{equalb}) have been found.
Choose $\alpha$ and $\beta_5$ approximating the $\alpha_v$ and the $\beta_{5,v}$ at the real and $2$-adic places~$v$ of~$k$ (this choice of places depends on $q$). 
Set $g_v:=\frac{\beta_{5,v}}{\tau\alpha_v}$ and $g:=\frac{\beta_5}{\tau\alpha}$.
We claim that~$f-g^2$ is a sum of~$4$ squares everywhere locally. 
To prove the claim at a real or $2$-adic place~$v$, we use that $f-g^2$ is close to~$f-g_v^2$, which is a sum of~$4$ squares by~(\ref{equalb}). If~$v$ is real, the key point is to verify that the nonnegativity of~$f-g_v^2$ implies that of~$f-g^2$. If $v$ is $2$-adic, we rely on an openness result, in the $2$-adic topology, for sums of~$4$ squares in $k_v(C_v)$, which we deduce from Kato's local criterion (see Proposition \ref{openness}).  At the other places of $k$, the claim is trivial.

It now follows from Kato's local-global principle that $f-g^2$ is a sum of $4$ squares in $k(C)$, hence that $f$ is a sum of $5$ squares in $k(C)$.

\subsection{Organization of the text}

Section \ref{secKato} gathers the consequences of Kato's theorems that we use.
The geometric studies leading to conditions \ref{intro} and \ref{iintro} in Theorem  \ref{main2} are carried out in Section \ref{secquadric}. 
The more arithmetic results of Section~\ref{secmodel} 
will allow us to verify condition \ref{iintro} at $p$-adic places when deducing Theorem \ref{main1} from Theorem \ref{main2}.
Finally, Section~\ref{secglobal} presents our global theorems. Theorem \ref{main2} is proven in \S\ref{parPourchet} following the strategy sketched in \S\ref{parsketch}. More concrete corollaries are derived in \S\ref{parcoro}. Applications to Pythagoras numbers, including a proof of Theorem~\ref{main1}, are given in~\S\ref{parPythagoras}.

\subsection{Acknowledgements}

I~thank Olivier Wittenberg for a useful conversation and Jean-Louis Colliot-Th\'el\`ene
and David Leep for helpful comments on the article.

This work was completed during a research stay at Leibniz University Hannover in May 2025, funded by a Mercator Fellowship of the German Research Foundation, within the RTG 2965 -- Project number 512730679.
I thank Stefan Schreieder for his hospitality on this occasion.

\subsection{Notation and conventions}

An \textit{algebraic variety} over a field $k$ is a separated $k$-scheme of finite type over $k$. A \textit{curve} is an algebraic variety of pure dimension $1$.

If $X$ is a scheme and $x\in X$, we let $\kappa(x)$ denote the residue field of $x$.

Let $k$ be a field with $2\in k^*$. If $a_1,\dots, a_n\in k^*$, we let $\llangle a_1,\dots, a_n\rrangle$ denote the Pfister quadratic form $\langle 1,-a_1\rangle\otimes\dots\otimes\langle 1,-a_n\rangle$. Beware that this quadratic form is rather denoted by $\llangle -a_1,\dots, -a_n\rrangle$ in \cite{Lam} (see \cite[X.1.1]{Lam}).
We also denote by ${\{a\}\in H^1(k,\Z/2)}$ the image of $a\in k^*$ by the boundary map $k^*\to H^1(k,\Z/2)$ of the Kummer exact sequence.

Let $q$ be a nondegenerate quadratic form of rank $r$ over $\R$. The \textit{signature} of~$q$ is the pair of integers $(s_+,s_-)$ such that whenever $q\simeq\langle a_1,\dots,a_r\rangle$ with ${a_i\in\R^*}$, exactly~$s_+$ (\resp $s_-$) of the $a_i$ are positive (\resp negative). The form $q$ is \textit{definite} if it is positive definite (of signature $(r,0)$) or negative definite (of signature~$(0,r)$). It has \textit{hyperbolic signature} if~$r\geq 1$ and it is of signature $(r-1,1)$ or~$(1,r-1)$.

Let $q$ be a quadratic form over a field $k$. Let $E$ be a $k$-vector space of dimension~$1$.
An element of $E^{\otimes 2}$ is said to be a \textit{represented by} $q$ if it is of the form $a\cdot (e\otimes e)$ for some $a\in k$ represented by $q$ and some $e\in E$. 
Observe that this notion depends on the presentation of the vector space $E^{\otimes 2}$ as a tensor square. If~$X$ is an algebraic variety over $k$ and $\cL$ is a line bundle on $X$, we say that a section~${\sigma\in H^0(X,\cL^{\otimes 2})}$ is represented by~$q$ at $x\in X$ if $\sigma|_x\in (\cL|_x)^{\otimes 2}$ is represented by~$q$. In particular, it makes sense to say that $\sigma$ is a square, a sum of $2$ squares, etc., at~$x$.
If moreover~${k=\R}$, we will say that $\sigma\in H^0(X,\cL^{\otimes 2})$ is \textit{nonnegative} (\resp \textit{positive}) if~$\sigma$ is a square (\resp a nonzero square) at all points $x\in X(\R)$.

Let $p$ be a prime number. A $p$\textit{-adic field} is a finite extension of $\Q_p$. We let $\cO_k$ denote the ring of integers of a $p$-adic field $k$.

We let $k_v$ denote the completion of a number field $k$ with respect to a place $v$ (real, complex, or $p$-adic for some prime number $p$). If $v$ is real, there is a canonical identification $k_v\simeq \R$.
If~$X$ is an algebraic variety over $k$, we set $X_v:=X\times_k k_v$. If~$q$ is a quadratic form over $k$, we let $q_v$ denote its extension of scalars from $k$ to~$k_v$.

\section{Consequences of Kato's theorems}
\label{secKato}

In this section, we draw consequences from \cite{Kato}. Our main new result is Proposition \ref{openness}.

\subsection{The local-global principle and the local criterion}

Recall the following consequence of the Merkurjev--Suslin theorem (see \cite[Theorem 12.1\,a)$\Leftrightarrow$c)]{MS}).

\begin{prop}
\label{lemMS}
Let $k$ be a field with $2\in k^*$. Fix $a,b,f\in k^*$. Then  $\llangle a,b\rrangle$ represents~$f$ in $k$ if and only if  $\{a\}\cdot\{b\}\cdot\{f\}\in H^3(k,\Z/2)$ vanishes.
\end{prop}

The next proposition appears in \cite[top of p.\,146]{Kato}. 

\begin{prop}
\label{localglobal}
Let $C$ be a geometrically connected smooth projective curve over a number field $k$. Fix~$a,b\in k^*$ and $f\in k(C)^*$. Then $\llangle a,b\rrangle$ represents $f$ in $k(C)$ if and only if it represents~$f$ in $k_v(C_v)$ for all places $v$ of $k$.
\end{prop}

\begin{proof}
In \cite[Theorem 0.8\,(2)]{Kato}, Kato proved that the restriction map
$$H^3(k(C),\Z/2)\to\prod_v H^3(k_v(C_v),\Z/2),$$
where $v$ runs over all places of $k$, is injective.
To conclude, combine this fact with Proposition~\ref{lemMS}.
\end{proof}

Here is another proposition which follows at once from Kato's work.

\begin{prop}
\label{local}
Let $C$ be a geometrically connected smooth projective curve over a $p$-adic field~$k$. Let~$\pi:X\to \Spec(\cO_k)$ be a proper regular model of $C$. Fix~$a,b\in k^*$ and~$f\in k(C)^*$.
Then~$\llangle a,b\rrangle$ represents $f$ in $k(C)$ if and only if it represents~$f$ in the henselization $k(C)^{\h}_x$ of $k(C)$ at all codimension $1$ points $x$ of $X$.
\end{prop}

\begin{proof}
Let $x\in X$ be a point of codimension $1$. Kato defines a residue map ${\partial_x:H^3(k(C),\Z/2)\to\Br(\kappa(x))[2]}$ (see \cite[(i)~p.\,149]{Kato} if $\kappa(x)$ has characteristic not~$2$ and \cite[(ii) p.\,150]{Kato} otherwise, both applied with $q=2$ and~$i=1$, noting that $H^2(\kappa(x),\Z/2(1))=\Br(\kappa(x))[2]$ in Kato's notation).

Kato proves in \cite[Proposition 5.2]{Kato} that residue maps induce a quasi-isomor\-phism between two complexes associated with the general and the special fiber of~$X$. Evaluated in degree $1$, this assertion exactly means that $\bigcap_x\Ker(\partial_x)=0$, where~$x$ runs over all codimension $1$ points of $X$. As $\partial_x$ factors through~$H^3(k(C)^{\h}_x,\Z/2)$, by its very construction,
%As noted in [Jannsen, Saito, Sato, p.10], one could use completion instead.
we deduce the injectivity of the restriction map 
\begin{equation}
\label{localinj}
H^3(k(C),\Z/2)\to\prod_x H^3(k(C)^{\h}_x,\Z/2),
\end{equation}
where $x$ runs over all codimension $1$ points of $X$.

The proposition follows from  the injectivity of~(\ref{localinj}) and Proposition~\ref{lemMS}.
\end{proof}

\subsection{An openness result}

The next statement is an application of Hensel's lemma.

\begin{lem}
\label{Hensel}
Let $(A,\km)$ be a
henselian discrete valuation ring.
Fix polynomials $f_1,\dots,f_m\in A[X_1,\dots, X_n]$ for some integers~$0\leq m\leq n$. Let $x\in A^n$  be such that~$f_j(x)=0$ for $1\leq j\leq m$ and $\big(\frac{\partial f_j}{\partial X_i}(x)\big)\in M_{n\times m}(\Frac(A))$ has rank $m$.

 For any integer $r\geq 0$, there exists an integer $s\geq 0$ with the following property.
 If ${g_1,\dots g_m\in A[X_1,\dots,X_n]}$ are such that $g_j-f_j\in\km^s\cdot A[X_1,\dots,X_n]$ for $1\leq j\leq m$, there exists $y\in A^n$ such that $g_j(y)=0$ for $1\leq j\leq m$ and~$y-x\in (\km^r)^n$.
\end{lem}

\begin{proof}
Let $\delta\in A$ be a nonzero minor of size $m$ of $\big(\frac{\partial f_j}{\partial X_i}(x)\big)$.
Let $t\geq 0$ be such that~${\delta\in \km^t\setminus\km^{t+1}}$. Choose $s:=\max(2t,t+r,t+1)$.

As $s\geq t+1$, the matrix $\big(\frac{\partial g_j}{\partial X_i}(x)\big)$ also has a minor of size $m$ in $\km^t\setminus\km^{t+1}$.
In addition, one has $g_j(x)=(g_j-f_j)(x)\in(\km^s)^n$ for~${1\leq j\leq m}$. 
By \cite[Lemma~5.10]{Artinapprox} applied with~${\ka=\km^{s-2t}}$, there exists $y\in A^n$ such that $g_j(y)=0$ for~$1\leq j\leq m$ and~$y-x\in (\km^{s-t})^n$. As $s-t\geq r$, this concludes the proof.
\end{proof}

Lemmas \ref{Krasner1} and \ref{Krasner2} below are geometric versions of Krasner's lemma. The only purpose of Lemma \ref{primitive} is to remove an unnecessary characteristic $0$ assumption from their statements; this is not used in our applications. Recall that a field extension~$l/k$ is said to be \textit{simple} if $l$ is generated over $k$ by a single element.

\begin{lem}
\label{primitive}
A finite field extension $l/k$ is simple if and only if $\dim_l\Omega^1_{l/k}\leq 1$.
\end{lem}

\begin{proof}
If $x\in l$ generates $l$ over $k$, then~$dx$ generates the $l$-vector space $\Omega^1_{l/k}$.
To prove the converse, we may assume that $l$ is purely inseparable over~$k$ (write $l/k$ as the composition of a separable and of a purely inseparable extension and apply the primitive element theorem in its form \cite[\S 6.10]{vdW}). In particular, the field~$k$ has positive characteristic $p$. Then, if $\dim_l\Omega^1_{l/k}\leq 1$, one has $[l:k\hspace{.05em}l^p]\mid p$ by \cite[Lemma~\href{https://stacks.math.columbia.edu/tag/07P2}{07P2}]{SP}, and $l/k$ is simple by \cite[Theorem 6]{BMacL}.
\end{proof}

\begin{lem}
\label{affineline}
Let $C$ be a smooth curve over a field $k$. Let $x\in C$ be a closed point. There exists a morphism $f:C\to \P^1_k$ that is \'etale at $x$ and that induces an isomorphism between $\kappa(f(x))$ and $\kappa(x)$.
\end{lem}

\begin{proof}
We may assume that $C$ is connected. Let $I$ be the ideal sheaf of $x$ in $C$.
As~$C$ is smooth, the coherent sheaf $\Omega^1_{C/k}$ is a line bundle on $C$. The exact sequence 
\begin{equation}
\label{Kahler}
I/I^2\xrightarrow{g\mapsto dg}\Omega^1_{C/k}|_x\to\Omega^1_{\kappa(x)/k}\to 0
\end{equation}
of $\kappa(x)$-vector spaces therefore shows that $\Omega^1_{\kappa(x)/k}$ has dimension~${\leq\,1}$ over $\kappa(x)$. By Lemma \ref{primitive}, the field $\kappa(x)$ is generated over $k$ by an ele\-ment~${t\in \kappa(x)}$ (so~$dt$ generates $\Omega^1_{\kappa(x)/k}$). Let~$f\in \cO_{C,x}$ be a lift of $t$. In view of~(\ref{Kahler}), after possibly replacing  $f$ with $f+g$ for some $g\in I_x$, we may assume that $df$ is nonzero at $x$. Our choices now imply that the morphism $f:C\to \P^1_k$  has the required properties.
\end{proof}
%MO371269 asks related question, with same motivation.

We call \textit{valued field} a field endowed with a nontrivial rank $1$ valuation (as in \cite[1.5.1]{BGR}). 
If $k$ is a complete valued field, we let $X^{\an}$ be the rigid analytification of an algebraic variety $X$ over $k$ (see \cite[9.3.4/2]{BGR}). 
Recall from \loccit that  the set underlying~$X^{\an}$ can be identified with the set of closed points of~$X$.

\begin{lem}
\label{Krasner1}
Let $C$ be a smooth curve over a complete valued field $k$. Fix $x\in C^{\an}$. There exists an open affinoid subvariety $x\in\Omega\subset C^{\an}$ (in the sense of \cite[9.3.1]{BGR}) such that for any $y\in \Omega$, there is a $k$-embedding of $\kappa(x)$ into $\kappa(y)$. 
\end{lem}

\begin{proof}
Let $f:C\to\P^1_k$ be as in Lemma \ref{affineline}. By \cite[7.3.3/5]{BGR}, there exist open affinoid subvarieties $x\in\Omega\subset C^{\an}$ and $f(x)\in\Omega'\subset(\P^1_k)^{\an}$ such that $f$ induces an isomorphism $f|_{\Omega}:\Omega\isoto\Omega'$.
%affinoid subdomains are open by 7.2.5/3
%would also need to check analytification preserves completed local rings.
By Krasner's lemma \cite[3.4.2/2]{BGR} applied in an appropriate affine chart of $\P^1_k$, we may shrink $\Omega$ and~$\Omega'$ so that there exists a $k$-embedding of $\kappa(f(x))$ into $\kappa(x')$, for any $x'\in \Omega'$. This concludes the proof.
\end{proof}

\begin{lem}
\label{Krasner2}
Let $C$ be a connected smooth projective curve over a complete valued field $k$.
Let $\cL$ be a line bundle on $C$. Let $\sigma\in H^0(C,\cL)$ be a nonzero section. There exists a neighborhood $U$ of $\sigma$ in $H^0(C,\cL)$ (for the topology induced by that of $k$) 
such that for any $\sigma'\in U$ and any $x'\in C$ with $\sigma'(x')=0$, there exists $x\in C$ with~$\sigma(x)=0$ and a $k$-embedding of $\kappa(x)$ into $\kappa(x')$.
\end{lem}

\begin{proof}
Let $(x_i)_{1\leq i\leq n}$ be the zeros of $\sigma$ on $C$. Use Lemma \ref{Krasner1} to find an open affinoid subvariety $x_i\in\Omega_i\subset C^{\an}$ such that for any $y_i\in\Omega_i$, there is a $k$-embedding of $\kappa(x_i)$ into $\kappa(y_i)$. As $C^{\an}$ is proper (in the sense of \cite[9.6.2]{BGR}), one can find open affinoid subvarieties $(\Omega'_j)_{1\leq j\leq m}$ of $C^{\an}$ not containing any of the $x_i$, such that~$C^{\an}$ is covered by the $\Omega_i$ and the $\Omega'_j$. After replacing each $\Omega'_j$ by an affinoid covering,
%finite by definition of affinoid covering
we may assume that $\cL^{\an}|_{\Omega'_j}$ is trivial. Fix trivializations $\varphi_j:\cL^{\an}|_{\Omega'_j}\isoto\cO_{\Omega'_j}$.

 We claim that, for all $1\leq j\leq m$,  there exists a neighborhood $U_j$ of $\sigma$ in $H^0(C,\cL)$ such that no $\sigma'\in U$ vanishes on $\Omega'_j$. Setting $U:=\cap_{j=1}^mU_j$ then concludes the proof.
 
We now fix $1\leq j\leq m$ and prove the claim. Choose a basis $(\sigma_1,\dots,\sigma_N)$ of~$H^0(C,\cL)$. Let $s,s_1,\dots,s_N\in\cO(\Omega'_j)$  be the images of $\sigma,\sigma_1,\dots,\sigma_N$ by $\varphi_j$. Let~$|\,.\,|_{\sup}$ be the supremum norm on the affinoid algebra $\cO(\Omega'_j)$
%8.2.1/2 ensures it has correct value.
(see \cite[6.2.1]{BGR}).
 Set $M_l:=|s_l|_{\sup}$ for $1\leq l\leq N$ and $M:=|\frac{1}{s}|_{\sup}$. 
 %If s does belong to any maximal ideal, then 1/s exists...
These choices ensure that if $a_1,\dots, a_N\in k$ are such that $|a_l|<\frac{1}{MM_l}$, then $s+\sum_{l=1}^Na_ls_l$ does not vanish anywhere on $\Omega'_j$, and hence neither does $\sigma+\sum_{l=1}^Na_l\sigma_l$. Define $U_j\subset H^0(C,\cL)$ to be the set of $\sigma+\sum_{l=1}^Na_l\sigma_l$ with $a_1,\dots, a_N\in k$ as above.
\end{proof}

We now reach the main goal of this section.

\begin{prop}
\label{openness}
Let $C$ be a 
connected smooth projective curve over a $p$-adic field~$k$. Let $\cL$ be a line bundle on $C$. 
Fix $a,b\in k^*$.
Let ${\Sigma\subset H^0(C,\cL^{\otimes 2})}$ be the set of sections represented by $\llangle a,b\rrangle$ at the generic point of $C$. Fix~${\sigma\in \Sigma\setminus\{0\}}$. If~$\llangle a,b\rrangle$ is isotropic in the residue fields of all the zeros of $\sigma$ of even multiplicity~$\geq2$, then~$\Sigma$ contains a $p$-adic neighborhood of $\sigma$ in $H^0(C,\cL^{\otimes 2})$.
\end{prop}

\begin{proof}
Fix a nonzero rational section $\alpha$ of $\cL$. It follows from the definitions that a section $\tau\in H^0(C,\cL^{\otimes 2})$ is represented by $\llangle a,b\rrangle$ at the generic point of $C$ if and only if~$\frac{\tau}{\alpha^2}\in k(C)$ is represented by $\llangle a,b\rrangle$ in $k(C)$. This remark allows us to transfer facts about representations of functions by $\llangle a,b\rrangle$ into statements about representations of sections by $\llangle a,b\rrangle$; we use it below without further comments.

Let $x\in C$ be a zero of odd multiplicity of $\sigma$.
As $\llangle a,b\rrangle$ represents $\sigma$ at the generic point of $C$, we deduce that $\llangle a,b\rrangle$ represents an element of odd valuation in $\cO_{C,x}$. It follows that $\llangle a,b\rrangle$ is isotropic in $\kappa(x)$ (see \eg \cite[Lemma 19.5]{EKM}). 
By hypothesis, the same holds for any zero of $\sigma$. 
Lemma \ref{Krasner2} therefore produces a neighborhood~$U\subset H^0(C,\cL^{\otimes 2})$ of $\sigma$ such that for any $\tau\in U$ and any $x'\in C$ with~$\tau(x')=0$, the form $\llangle a,b\rrangle$ is isotropic in $\kappa(x')$.

 We claim that for any $\tau\in U$ and any $x\in C$, the form $\llangle a,b\rrangle$ represents~$\tau$ in the henselization~$k(C)^{\h}_x$ of $k(C)$ at $x$. If $x$ is not a zero of $\tau$, then $\llangle a,b\rrangle$ represents~$\tau(x)$ in~$\kappa(x)$ (see \cite[I.3.5 and VI.2.12]{Lam}), and hence represents~$\tau$ in~$k(C)^{\h}_x$ by henselianity (use \cite[Lemma~\href{https://stacks.math.columbia.edu/tag/0H74}{0H74}]{SP}). If $x$ is a zero of $\tau$, then $\llangle a,b\rrangle$ is isotropic in $\kappa(x)$ by choice of $U$, hence is isotropic in $k(C)^{\h}_x$ by another application of \cite[Lemma~\href{https://stacks.math.columbia.edu/tag/0H74}{0H74}]{SP}. It follows that $\llangle a,b\rrangle$ represents $\tau$ in $k(C)^{\h}_x$ (see \mbox{\cite[I.3.4]{Lam}}).

Let~$\pi:X\to \Spec(\cO_k)$ be a proper regular model of $C$. Let $x$ be the generic point of an irreducible component of the special fiber of $X$. Since the form $\llangle a,b\rrangle$ represents~$\sigma$ in~$k(C)$,
the form $\llangle a,b\rrangle\perp\langle-\frac{\sigma}{\alpha^2}\rangle$ is isotropic over $k(C)$, hence also over the henselization $k(C)^{\h}_x$ of $k(C)$ at $x$. Let $c\in (k(C)^{\h}_x)^*$ be such that $c\cdot\frac{\sigma}{\alpha^2}\in\cO_{X,x}^{\h}$. By Lemma \ref{Hensel} (applied over~$A=\cO_{X,x}^{\h}$ with~$m=1$ and $f_1$ equal to the quadratic form $c\cdot\big(\llangle a,b\rrangle\perp\langle-\frac{\sigma}{\alpha^2}\rangle\big)$), there exists a neighborhood $V_x\subset H^0(C,\cL^{\otimes 2})$ of ~$\sigma$ such that 
$c\cdot\big(\llangle a,b\rrangle\perp\langle-\frac{\tau}{\alpha^2}\rangle\big)$ is isotropic in $k(C)^{\h}_x$ for all $\tau\in V_x$. In view of \cite[I.3.4]{Lam}, the form $\llangle a,b\rrangle$ represents $\tau$ in~$k(C)_x^{\h}$ for all $\tau\in V_x$.

To conclude, define $W:=U\cap \bigcap_xV_x$ (where $x$ runs over all the generic points of the irreducible components of the special fiber of $X$). It is a neighborhood of~$\sigma$ in~$H^0(C,\cL^{\otimes 2})$. Our choices imply that for all $\tau\in W$, the form $\llangle a,b\rrangle$ represents~$\tau$ in~$k(C)_x^{\h}$ (for all points $x\in X$ of codimension $1$).  By Proposition \ref{local}, the form~$\llangle a,b\rrangle$ also represents $\sigma$ in $k(C)$.
\end{proof}

\begin{rems}
(i) The hypothesis on the zeros of $\sigma$ in Proposition \ref{openness} cannot be removed. To see it, set $a=b=-1$, take $k:=\Q_2$ and $C:=\P^1_{\Q_2}$ with homogeneous coordi\-nates~$[X:Y]$, and choose~$\cL:=\cO_{C}(1)$. Then $X^2$ is represented by~$\llangle-1,-1\rrangle$ in~$k(C)$, but $X^2-a^2Y^2$ is not represented by $\llangle-1,-1\rrangle$ in $k(C)$, for any~$a\in\Q_2^*$ (by \cite[IX.2.3]{Lam} and since $-1$ is not a sum of $3$ squares in $\Q_{2}$).

(ii) Proposition \ref{openness} applies, in particular, when the zero locus of $\sigma$ is reduced.
\end{rems}

\section{Representing sections of line bundles by quadratic forms}
\label{secquadric}
 
 Let $C$ be a connected smooth projective curve over a field $k$ and let $q$ be a nondegenerate quadratic form over $k$. In this section, we find necessary and sufficient conditions for a section of the square of a line bundle over $C$ to be be represented by $q$ over~$k(C)$, when the quadratic form is either isotropic (see Proposition \ref{condition2prop}), or definite over the field of real numbers (see Proposition \ref{condition1prop}).

\subsection{Geometry of a quadric bundle}
\label{parquadrics}

Let $k$ be a field with $2\in k^*$. 
Choose elements $b_1,\dots,b_{r-2}\in k^*$ for some $r\geq 3$.
 Let~$C$ be a connected smooth projective curve over $k$. 
 Let $\cL$ be a line bundle on~$C$. Fix a nonzero section $\sigma\in H^0(C,\cL^{\otimes 2})$ with reduced zero locus.

Define $\P:=\P_C(\cL\oplus\cO_C^{\oplus r})$, where the projectivization is taken in Grothendieck's sense (so $\P$ parametrizes rank $1$ quotients of $\cL\oplus\cO_C^{\oplus r}$). Let $f:\P\to C$ be the structural morphism. One then computes that ${H^0(\P,\cO_\P(1))=H^0(C,\cL)\oplus H^0(C,\cO_C)^{\oplus r}}$ (define $T_1,\dots,T_r\in H^0(\P,\cO_\P(1))$ to be generators of the last $r$ factors) and that $H^0(\P,\cO_\P(1)\otimes f^*\cL^{-1})= H^0(C,\cO_C)\oplus H^0(C,\cL^{-1})^{\oplus r}$
(let $T_0\in H^0(\P,\cO_\P(1)\otimes f^*\cL^{-1})$ be a generator of the first factor). Define
\begin{equation}
\label{eqX}
X:=\{\sigma\, T_0^2=\sum_{i=1}^{r-2}b_iT_i^2+T_{r-1}T_r\}\subset \P.
\end{equation}
Then $f|_X:X\to C$ endows $X$ with the structure of a quadric bundle. Its fibers are rank $r+1$ quadrics, except over the zero locus of $\sigma$, where they are rank $r$ quadrics. That $\sigma$ has reduced zero locus implies that $X$ is smooth.

We also consider~$\P':=\P_C(\cL\oplus\cO_C^{\oplus (r-1)})$ and let $g:\P'\to C$ be the structural morphism. As above, introduce the coordinate system on ~$\P'$ given by the canonical sections $U_0\in H^0(\P',\cO_{\P'}(1)\otimes g^*\cL^{-1})$ and $U_1,\dots, U_{r-1}\in H^0(\P',\cO_{\P'}(1))$.

Let $\iota:C\to \P$ be the section of $f$ associated with the projection ${\cL\oplus\cO_C^{\oplus r}\to\cO_C}$ onto the last factor. Let $\widetilde{\P}\to\P$ be the blow-up of $\P$ along the image $\iota(C)$ of $\iota$. 
Projection from $\iota(C)$ yields a birational map~$\P\dashrightarrow\P'$ which is resolved by blowing up $\iota(C)$, and hence gives rise to a morphism $\wP\to\P'$ 
%which is a $\P^1$-bundle 
(these facts are easily checked locally, over open subsets of $C$ where $\cL$ can be trivialized).

As the image of $\iota$ is included in $X$, the strict transform $\wX$ of $X$ in $\wP$ is the blow-up of $X$ along the image of $\iota$. Let $p:\wX\to X$ be the blow-up morphism. Set
$$Y:=\{U_{r-1}=\sigma\, U_0^2-\sum_{i=1}^{r-2}b_iU_i^2=0\}\subset\P'.$$
Then the projection morphism $q:\wX\to \P'$ identifies with the blow-up of $Y$ (again, this can be checked locally on $C$). 
The following diagram summarizes the situation.

\begin{equation*}
\begin{aligned}
\label{diagquad}
\xymatrix
@R=2ex
@C=.75cm{
&\wP\ar@/^2.3pc/[rrd]\ar@{}[r]|-*[@]{\supset}\ar[ld]&\wX\ar^p[ld]\ar_q[rd]& \\
\P\ar_f@/_0.4pc/[rrd]\ar@{}[r]|-*[@]{\supset}&X\ar[rd]&&\P'\ar^g[ld]\\
&&C&
}
\end{aligned}
\end{equation*}

Let $E$ and $F$ be the exceptional divisors of $p$ and $q$ respectively.

\begin{lem}
\label{lemlb}
One has $p^*(\cO_{\P}(1)|_X)\simeq\cO_{\wX}(2E+F)$ and $q^*\cO_{\P'}(1)\simeq\cO_{\wX}(E+F)$.
\end{lem}

\begin{proof}
The divisor $F\subset \wX$ is the strict transform in $\wX$ of ${Z:=\{T_{r-1}=0\}\subset X}$. The very definition of $Z$ shows that $\cO_X(Z)\simeq\cO_{\P}(1)|_X$. In addition, as the multiplicity of $Z$ along the image of $\iota$ is equal to $2$, the inverse image by $p$ of the Cartier divisor $Z\subset X$ is equal to $2E+F$. The first equality follows.

The divisor $E\subset \wX$ is the strict transform in $\wX$ of $W:=\{U_{r-1}=0\}\subset \P'$, so that $\cO_{\P'}(W)\simeq\cO_{\P'}(1)$. As $W$ is smooth along $Y$, the inverse image of $W\subset \P'$ by~$q$ is the Cartier divisor $E+F$. This completes the proof of the lemma.
\end{proof}

\subsection{The isotropic case}

The next proposition is the main result of this section.

\begin{prop}
\label{condition2prop}
Let $k$ be an infinite field with $2\in k^*$.
Choose ${a_1,\dots,a_r\in k^*}$ for some $r\geq 3$. Assume that $\langle a_1,\dots a_r\rangle\simeq \langle b_1,\dots, b_{r-2}\rangle\perp h$, where ${b_1,\dots, b_{r-2}\in k^*}$ 
and $h$ is the hyperbolic plane $(x,y)\mapsto xy$.
Let $C$ be a connected smooth projective curve over~$k$. Fix $\cL$, $\cM$ and $\cA$ in $\Pic(C)$, with $\cA$ ample. For~${\sigma\in H^0(C,\cL^{\otimes 2})}$ nonzero with reduced zero locus, the following assertions are \mbox{equivalent}.
\begin{enumerate}[label=(\roman*)] 
\item
\label{iquad}
For some $\cN\in\Pic(C)$, there exist a nonzero section $\alpha\in H^0(C,\cM\otimes\cN^{\otimes 2})$, and sections $\beta_{i}\in H^0(C,\cL\otimes\cM\otimes\cN^{\otimes 2})$ no two of which have a common zero, such that
$\sigma\alpha^2=\sum_{i=1}^ra_i\beta_{i}^2$.
 \item
 \label{iiquad}
 Assertion \ref{iquad} holds with $\cN=\cA^{\otimes l}$ for all large enough integers $l$.
 \item
 \label{iiiquad}
 There exist $\cP\in\Pic(C)$ and a divisor~$\Delta$ on $C$ with $\cO_C(\Delta)\simeq\cL\otimes\cM\otimes\cP^{\otimes 2}$, such that~$\sigma$ is nonzero and represented by $\langle b_1,\dots, b_{r-2}\rangle$ at all points of the support of $\Delta$. 
\end{enumerate}
\end{prop}

\begin{proof}
We use the notation of \S\ref{parquadrics}. 
Solutions to 
\begin{equation}
\label{alphabeta}
\sigma\alpha^2=\sum_{i=1}^ra_i\beta_{i}^2
\end{equation}
with $\alpha\in H^0(C,\cM\otimes\cN^{\otimes 2})$ and $\beta_{i}\in H^0(C,\cL\otimes\cM\otimes\cN^{\otimes 2})$ without common zero are in bijection with solutions to 
\begin{equation}
\label{alphagamma}
\sigma\alpha^2=\sum_{i=1}^{r-2}b_i\gamma_{i}^2+\gamma_{r-1}\gamma_r
\end{equation}
with $\alpha\in H^0(C,\cM\otimes\cN^{\otimes 2})$ and $\gamma_{i}\in H^0(C,\cL\otimes\cM\otimes\cN^{\otimes 2})$ without common zero (making use of an isomorphism $\langle a_1,\dots a_r\rangle\simeq \langle b_1,\dots, b_{r-2}\rangle\perp h$).
In turn, they are in bijection with sections $s:C\to X$ of $f|_X:X\to C$ with $s^*\cO_{\P}(1)\simeq\cL\otimes\cM\otimes\cN^{\otimes 2}$.

That \ref{iiquad} implies \ref{iquad} is clear.
Assume that \ref{iquad} holds. Let  $\alpha$ and $\beta_i$ be as in~\ref{iquad} for some $\cN\in\Pic(C)$.
Consider the associated $\alpha$ and $\gamma_i$ as in (\ref{alphagamma}). As the $\beta_i$ do not have a common zero, neither do the $\gamma_i$.
Consequently, after applying a general element of the special orthogonal group of $\langle b_1,\dots, b_{r-2}\rangle\perp h$ (using that this group is $k$\nobreakdash-unirational, 
%actually k-rational, by Cayley parametrization (e.g. https://arxiv.org/pdf/0809.4481, Remark 4.1).
see \cite[V, Theorem 18.2]{Borel}, and that $k$ is infinite), one can assume that $\gamma_r$ does not vanish on any zero of $\sigma$ or of $\alpha$.
Then~$\Delta:=\{\gamma_r=0\}$ satisfies \ref{iiiquad} with $\cP=\cN$, as equation~(\ref{alphagamma}) shows.

To conclude, we show that \ref{iiiquad} implies \ref{iiquad}.
Let $\cP$ and $\Delta$ be as in \ref{iiiquad}. After modifying $\Delta$ by a multiple of $2$ (and $\cP$ by a square), we may assume that $\Delta$ is effective and reduced. Then \ref{iiiquad} shows the existence of a section $s'_{\Delta}:\Delta\to Y$ of~$q:\P'\to C$ over $\Delta$. This section corresponds to a surjection 
\begin{equation}
\label{surjection}
(\cL\oplus\cO_C^{\oplus r})|_{\Delta}\to\cO_{\Delta}\simeq (\cL\otimes \cM\otimes \cP\otimes \cA^{\otimes l})|_{\Delta},
\end{equation}
where we chose a trivialization of $\cL\otimes \cM\otimes \cP\otimes \cA^{\otimes l}$ on $\Delta$. 
The ampleness of $\cA$ implies that, for $l$ large enough, one can lift (\ref{surjection}) to a surjection 
\begin{equation}
\label{lift}
\cL\oplus\cO_C^{\oplus r}\to\cL\otimes \cM\otimes \cP\otimes \cA^{\otimes l},
\end{equation}
corresponding to a section $s':C\to\P'$ of $g$ with $(s')^*\cO_{\P'}(1)\simeq\cL\otimes \cM\otimes \cP\otimes \cA^{\otimes l}$. For $l$ large enough, choosing the lift (\ref{lift}) general ensures that $s'(C)$ meets $Y$ transversally along $s'(\Delta)$ (and at no other point), and that $s'(C)\not\subset\{U_0=0\}\cup\{U_{r-1}=0\}$.

Let $\ws:C\to\wX$ be the strict transform of $s'$, and define $s:=\ws\circ p:C\to X$, so~$s$ is a section of $f|_X$. As $s'(C)$ meets $Y$ transversally along $s'(\Delta)$ and at no other point, we see that $\ws(C)$ meets $F$  transversally along $\ws(\Delta)$ and at no other point, and hence that
$\ws\hspace{.09em}^*\cO_{\wX}(F)=\cO_C(\Delta)\simeq\cL\otimes\cM\otimes\cP^{\otimes 2}$. It then follows from the second equality of Lemma \ref{lemlb} (also using $(s')^*\cO_{\P'}(1)\simeq\cL\otimes \cM\otimes \cP\otimes \cA^{\otimes l}$), that~$\ws\hspace{.09em}^*\cO_{\wX}(E)\simeq\cP^{\otimes -1}\otimes\cA^{\otimes l}$. The second equality of Lemma \ref{lemlb} now implies that~$s^*(\cO_{\P}(1)|_X)\simeq\cL\otimes\cM\otimes\cA^{\otimes 2l}$. As was explained as the beginning of the proof, the section $s$ corresponds to a solution $(\alpha, (\beta_i)_{1\leq i\leq r})$ to (\ref{alphabeta}) with $\cN=\cA^{\otimes l}$.

As we ensured that  $s'(C)\not\subset\{U_0=0\}\cup\{U_{r-1}=0\}$, one has~$\alpha\neq 0$.  The~$\beta_i$ do not have a common zero (as $\alpha$ cannot vanish on such a common zero, equation~(\ref{alphabeta}) shows that it would be a multiple zero of $\sigma$).
Consequently, after applying a general element of the special orthogonal group of $\langle a_1,\dots, a_{r}\rangle$ (using that this group is $k$\nobreakdash-unirational, see \cite[V, Theorem 18.2]{Borel}, and that $k$ is infinite), one can assume that no two of the $\beta_i$ have a common zero. This completes the proof of \ref{iiquad}.
\end{proof}

\subsection{The definite case}

The next proposition will serve as a substitute to Proposition \ref{condition2prop} for definite quadratic forms over the reals. For a connected smooth projective curve $C$ over $\R$, we let $\cl_{\R}:\Pic(C)\to H^1(C(\R),\Z/2)$ denote the Borel--Haefliger cycle class map defined as in \S\ref{parth2}.

\begin{prop}
\label{condition1prop}
Fix $r\geq 2$. Let~$C$ be a connected smooth projective curve over~$\R$. Fix $\cL$, $\cM$ 
 and $\cA$ in $\Pic(C)$, with $\cA$ ample. For~${\sigma\in H^0(C,\cL^{\otimes 2})}$ nonzero with reduced zero locus, the following assertions are \mbox{equivalent}.
\begin{enumerate}[label=(\roman*)] 
\item
\label{iquaR}
For some $\cN\in\Pic(C)$, there exist a nonzero section $\alpha\in H^0(C,\cM\otimes\cN^{\otimes 2})$, and sections $\beta_{i}\in H^0(C,\cL\otimes\cM\otimes\cN^{\otimes 2})$ without a common real zero, such that~$\sigma\alpha^2=\sum_{i=1}^r\beta_{i}^2$.
 \item
 \label{iiquaR}
 Assertion \ref{iquad} holds with $\cN=\cA^{\otimes l}$ for all large enough integers $l$.
 \item
 \label{iiiquaR}
The section $\sigma$ is nonnegative and $\cl_{\R}(\cM)=0$.
\end{enumerate}
\end{prop}

\begin{proof}
It is obvious that \ref{iiquaR} implies \ref{iquaR}. Assume that \ref{iquaR} holds. The nonnegativity of $\sigma$ is clear from the equation $\sigma\alpha^2=\sum_{i=1}^r\beta_{i}^2$, as $\alpha\neq 0$. Suppose for contradiction that $\cl_{\R}(\cM)=\cl_{\R}(\cM\otimes\cN^{\otimes 2})\neq 0$. Then $\alpha\in H^0(C,\cM\otimes\cN^{\otimes 2})$ has at least one real zero. At such a real zero, the equation $\sigma\alpha^2=\sum_{i=1}^r\beta_{i}^2$ shows that all the~$\beta_i$ must vanish. This contradicts our hypothesis and proves \ref{iiiquaR}.

Assume now that \ref{iiiquaR} holds. Let $\tau$ be a nonzero rational section of $\cL$. The rational function ${f:=\frac{\sigma}{\tau^2}\in\R(C)}$ is nonnegative. By a theorem of Witt (see \cite[I~p.\,4]{Witt}), one can therefore write $f=g^2+h^2$ with $g,h\in\R(C)$. By another theorem of Witt (see \cite[III~p.\,4]{Witt}), the hypothesis that $\cl_{\R}(\cM)=0$ implies that $\cM$ has a nonzero rational section $\delta$ with no real zero and no real pole.

Then $\sigma\delta^2=(g\delta\tau)^2+(h\delta\tau)^2$.
As the rational section~$\sigma\delta^2$ of $\cL^{\otimes 2}\otimes \cM^{\otimes 2}$ has no real poles, neither $g\delta\tau$ nor $h\delta\tau$ have real poles. One can therefore choose an effective divisor~$D$ on $C$, whose support has no real points, with~${D+\ddiv(g\delta\tau)}$ and~${D+\ddiv(h\delta\tau)}$ effective. Choose $l$ large enough, so $\cA^{\otimes l}(-D)$ is globally generated. Let~$(\varepsilon_i)_{1\leq i\leq n}$ be a basis of $H^0(C,\cA^{\otimes l}(-D))$. Then $\varepsilon:=\sum_{i=1}^n\varepsilon_i^2$, viewed as an element of~$H^0(C,\cA^{\otimes 2l})$ vanishing on $D$, has no real zero. After maybe changing~$\varepsilon_1$, we may assume that~$\varepsilon\neq 0$.

 Set $\alpha:=\delta\varepsilon$, and $(\beta_1,\dots, \beta_r):=(g\delta\varepsilon\tau,h\delta\varepsilon\tau,0,\dots, 0)$. These choices ensure that
 \begin{equation}
 \label{eqagain}
 \sigma\alpha^2=\sum_{i=1}^r\beta_{i}^2
 \end{equation}
 and that $\alpha\neq 0$.
 As neither $\delta$ nor $\varepsilon$ have a real zero, a common real zero of $\beta_1$ and~$\beta_2$ would have to be a multiple zero of $\sigma$ (by (\ref{eqagain})), which is impossible. So $\beta_1$ and~$\beta_2$ have no common real zero, which proves \ref{iiquaR}.
\end{proof}

\section{Curves over finite and \texorpdfstring{$p$}{p}-adic fields}
\label{secmodel}

 The aim of this section is to prove Proposition \ref{propmiracle}, which plays a crucial role in our results on Pythagoras numbers. 
 
In \S\ref{parCFT}, we use global class field theory to describe the subgroup of the Picard group of a reduced projective curve over a finite field that is generated by classes of smooth closed points of degree divisible by $d$ (see Proposition~\ref{propevenF}). In \S\ref{pareven}, we deduce an analogous result for smooth projective curves over $p$-adic fields (see Proposition~\ref{propeven}). We use it in \S\ref{parvalues} to control the values of a section of the square of line bundle on a $p$-adic curve, and in particular to prove Proposition \ref{propmiracle}.

\subsection{Line bundles on curves over finite fields}
\label{parCFT}

In this paragraph, we fix a finite field $\F_q$ of cardinality $q$ and an algebraic closure $\oF_q$ of $\F_q$. For $r\geq 1$, we let~${\F_q\subset \F_{q^r}\subset\oF_q}$ be the subextension that is of degree $r$ over $\F_q$.
We first recall a standard consequence of the Lang--Weil estimates.

\begin{lem}
\label{lemLW}
Let $X$ be a geometrically integral variety of dimension $n\geq 1$ over~$\F_q$. For~$r\geq 1$, let $M_X(r)$ be the number of closed points of $X$ that are of degree~$r$ over~$\F_q$. Then there exists $K\geq 0$ such that
\begin{equation}
\label{LWclosed}
|M_X(r)-\frac{q^{nr}}{r}|\leq K\,\frac{q^{(n-\frac{1}{2})r}}{r}\textrm{\hspace{1em} for all }r\geq 1.
\end{equation}
In particular, one has $M_X(r)>0$ for all large enough integers $r\geq 1$.
\end{lem}

\begin{proof}
Let $N_X(r)$ be the cardinality of $X(\F_{q^r})$. 
One has $N_X(r)=\sum_{s|r}sM_X(s)$. The M\" obius inversion formula (for which see \eg \cite[Theorem~2.9]{Apostol}) shows that
\begin{equation}
\label{Mobiusformula}
M_X(r)=\frac{1}{r}\sum_{s\mid r}\mu(\frac{r}{s})N_X(s),
\end{equation}
where $\mu$ is the M\" obius function.
On the other hand, it follows from the Lang--Weil estimates \cite[Theorem 1]{LangWeil} that there exists a constant $K'\geq 0$ such that 
\begin{equation}
\label{LW}
|N_X(r)-q^{nr}|\leq K'\,q^{(n-\frac{1}{2})r}
\end{equation}
for all $r\geq 1$. We deduce from (\ref{Mobiusformula}) and (\ref{LW}) that
\begin{equation}
\label{Mobius}
|M_X(r)-\frac{q^{nr}}{r}|\leq K'\,\frac{q^{(n-\frac{1}{2})r}}{r}+\frac{1}{r}\sum_{s\leq \lfloor\frac{r}{2}\rfloor}N_X(s).
\end{equation}
In view of (\ref{LW}), one has $N_X(r)\leq (1+K')\,q^{nr}$ for all $r\geq 1$. We deduce at once that~$\sum_{s\leq\lfloor\frac{r}{2}\rfloor}N_X(s)\leq (1+K')\,q^{n(\frac{r}{2}+1)}$. Combining this fact with (\ref{Mobius}) yields the estimate (\ref{LWclosed}) for an appropriate constant $K$.
\end{proof}

Let $X$ be a connected smooth projective curve over a finite field $\F_q$.
Let $D\subset X$ be an effective divisor.
We define~$\Pic(X,D)$ to be the set of isomorphism classes of pairs~$(\cL,\varphi)$, where $\cL$ is a line bundle on $X$ and~${\varphi:\cO_D\isoto\cL|_D}$ is a trivialization of~$\cL$ on $D$.
We endow~$\Pic(X,D)$ with the group structure induced by tensor product. 
%It is the \textit{ray class group} with modulus~$D$.
With any closed point~${x\in X\setminus D}$, we associate the class~${[x]\in\Pic(X,D)}$ of the line bundle~$\cO_X(x)$ equipped with the restriction to $D$ of its canonical trivialization on $X\setminus \{x\}$.
For any~${d\geq 0}$, we let~$\Pic(X,D)^d$ be the subgroup of~$\Pic(X,D)$ consisting of those~$(\cL,\varphi)$ such that  the degree of $\cL$ over $\F_q$ is a multiple of $d$.
As there are only finitely many isomorphism classes of degree $0$ line bundles on $X$, the group~$\Pic(X,D)^0$ is finite.

Let $\pi_1^{\ab}(X\setminus D)$ be the abelianization of the \'etale fundamental group of $X\setminus D$, and let $\pi_1^{\ab}(X,D)$ be its quotient classifying those abelian coverings whose ramification is bounded by $D$ in the sense of \cite[(8.6)]{BKS}. The pushforward by the structural morphism~$X\setminus D\to\Spec(\F_q)$ yields a morphism $\deg:\pi_1^{\ab}(X,D)\to \Gal(\oF_q/\F_q)\simeq\widehat{\Z}$ with kernel~$\pi_1^{\ab}(X,D)^0$.
%By global class field theory, t
There is a commutative diagram with exact rows
\begin{align*}
\begin{aligned}
\xymatrix@R=3ex{
0 \ar[r] & \Pic(X,D)^0 \ar[r] \ar^{\rotatebox{90}{$\sim$}}[d]& \Pic(X,D) \ar^{\hspace{1.5em}\deg}[r] \ar[d]^{\rho_{(X,D)}} & \Z\ar[d]\\
0 \ar[r] & \pi_1^{\ab}(X,D)^0\ar[r] & \pi_1^{\ab}(X,D) \ar^{\hspace{1.5em}\deg}[r] & \widehat{\Z}
}
\end{aligned}
\end{align*}
whose middle vertical arrow is the reciprocity map defined by $\rho_{(X,D)}([x])=\Frob_x$ (where $\Frob_x$ is the Frobenius of the closed point $x\in X\setminus D$), whose left vertical arrow is an isomorphism, and whose right vertical arrow is the canonical inclusion. The above assertion combines the main theorems of global class field theory for function fields, first proven by Hasse and Witt \cite{WittCFT}; it appears exactly in the above form \eg in \cite[(8.8), Lemma 8.4, Theorem 8.5]{BKS}.

\begin{lem}
\label{lemCFT}
Let $X$ be a connected smooth projective curve over $\F_q$. Let~$D\subset X$ be an effective divisor. For any $d\geq 1$, the group $\Pic(X,D)^d$ is gen\-er\-at\-ed by the classes of those closed points of $X\setminus D$ whose degree over $\F_q$ is a multiple~of~$d$.
\end{lem}

\begin{proof}
After replacing $\F_q$ with its algebraic closure in $X$, we may assume that $X$ is geometrically connected. Let $G_d\subset \Pic(X,D)^d$ be the subgroup generated by 
the classes of those closed points of $X\setminus D$ whose degree over $\F_q$ is a multiple~of~$d$.
Consider the finite group $G_d^0:=G_d\cap\Pic(X,D)^0$. It follows from Lemma \ref{lemLW} that~$\deg(G_d)=d\mkern 1mu\Z$, so there is a short exact sequence $0\to G_d^0\to G_d\xrightarrow{\deg}d\mkern 1mu\Z\to 0$.

Define $H_d^0:=\rho_{(X,D)}(G_d^0)$ and let $H_d\subset \pi_1^{\ab}(X,D)$ be the closure of $\rho_{(X,D)}(G_d)$. One then has a short exact sequence $0\to H_d^0\to H_d\xrightarrow{\deg}d\mkern 1mu\widehat{\Z}\to 0$. As $H_d$ is a closed subgroup of finite index in $\pi_1^{\ab}(X,D)$, it corresponds to an abelian connected finite \'etale cover $\mu_d:\whX_d\to X$. The degree $\delta_d$ of $\mu_d$ equals the index of $H_d$ in~$\pi_1^{\ab}(X,D)$. Since~$\deg(H_d)=d\mkern 1mu\widehat{\Z}$, the algebraic closure of $\F_q$ in $\whX_d$ is isomorphic to~$\F_{q^d}$.
It follows that $\mu_d$ factors as a composition $\whX_d\to X_{\F_{q^d}}\to X$.

By Lemma \ref{lemLW}, the number of closed points of $X$ of degree $rd$ over $\F_q$ grows as~$\frac{q^{rd}}{rd}$ when $r$ goes to infinity. For any such point $x\in X$, one has $\Frob_x\in H_d$ by definition of $H_d$. This means that $x$ splits completely in $\whX_d$, \ie that the preimage of $x$ in $\whX_d$ consists of $\delta_d$ points of degree $rd$ over $\F_q$. Consequently, the number of points of $\whX_d$ of degree $rd$ over $\F_q$ grows at least as~$\delta_d\frac{q^{rd}}{rd}$ when $r$ goes to infinity. On the other hand, it follows from Lemma \ref{lemLW} (applied to the variety $\whX_d$ over $\F_{q^d}$) that this number grows as $\frac{q^{rd}}{r}$ when $r$ goes to infinity. We deduce that $\delta_d\leq d$.

The morphism $\whX_d\to X_{\F_{q^d}}$
must therefore be an isomorphism (as the degrees of the connected finite \'etale covers 
$\whX_d$ and $X_{\F_{q^d}}$ of $X$ 
are $\delta_d$ and $d$). They are thus associated with the same finite index subgroup of $\pi_1^{\ab}(X,D)$. Taking the inverse image of these subgroups by $\rho_{(X,D)}$ yields the desired equality $G_d=\Pic(X,D)^d$.
\end{proof}

Let $X$ be a reduced projective curve over $\F_q$. Let $(X_i)_{1\leq i\leq m}$ be the irreducible components of $X$. 
We define $\Pic(X)^d\subset \Pic(X)$ to be the subgroup of line bundles~$\cL$ on $X$ such that the degree of $\cL|_{X_i}$ over $\F_q$ is divisible by $d$ for $1\leq i\leq m$.

\begin{prop}
\label{propevenF}
Let $X$ be a reduced projective curve over $\F_q$. Let $\Xi\subset X$ be a finite subset of closed points.
Fix $d\geq 1$. Then~$\Pic(X)^{d}$ is generated by the classes of those smooth closed points of $X\setminus\Xi$ whose degree over $\F_q$ is a multiple of $d$. 
\end{prop}

\begin{proof}
Let $(X_i)_{1\leq i\leq m}$ be the irreducible components of $X$. For $1\leq i\leq m$, we define $X_i':=X_1\cup\dots\cup X_i\subset X$ and we endow it with its reduced structure. Fix~$\cL\in \Pic(X)^d$ and let $1\leq i\leq m+1$ be maximal such that $\cL|_{X_{i-1}'}$ is trivial. To show that $\cL$ belongs to the subgroup generated by those smooth closed points of~$X\setminus\Xi$ whose degree over $\F_q$ is a multiple of $d$, we argue by decreasing induction on $i$ (the case $i=m+1$ being trivial).

Consider the scheme-theoretic intersection $Z_i:=X_i\cap X'_{i-1}$ and the normalization morphism $\nu_i:\wX_i\to X_i$. Let $\cI_i\subset \cO_{X_i}$ be the coherent ideal sheaf defined as the annihilator of the cokernel of~$\cO_{X_i}\to(\nu_i)_*\cO_{\wX_i}$ (the \textit{conductor} of $\nu_i$).
Let $\cJ_i\subset\cI_i$ be the ideal sheaf of a finite closed subscheme of $X_i$ containing $Z_i$ and $\Xi\cap X_i$.
Since~$\cJ_i$ kills the cokernel of~$\cO_{X_i}\to(\nu_i)_*\cO_{\wX_i}$, it is also an ideal in the sheaf of rings~$(\nu_i)_*\cO_{\wX_i}$. 
Let $D_i\subset X_i$ and $\wD_i\subset\wX_i$ be the subschemes defined by $\cJ_i$. 
We get a commutative diagram:
\begin{equation}
\begin{aligned}
\label{diagconductor}
\xymatrix
@R=0.5cm 
{
0\ar[r]&\cJ_i\ar[r]\ar@{=}[d]& \cO_{X_i}\ar[r]\ar[d]&\cO_{D_i}\ar[d]\ar[r] &0\\
0 \ar[r]&\cJ_i\ar[r]&(\nu_i)_*\cO_{\wX_i}\ar[r]&  (\nu_i)_*\cO_{\wD_i} \ar[r] &0.}
\end{aligned}
\end{equation} 

Let $\Pic(X_i,D_i)$ (\resp $\Pic(\wX_i,\wD_i)$) be the groups of isomorphism classes of line bundles on $X_i$ (\resp $\wX_i$) endowed with a trivialization on $D_i$ (\resp $\wD_i$).
Let $\cA_i$ (\resp $\widetilde{\cA}_i$) be the \'etale sheaf of invertible functions on $X_i$ (\resp $\wX_i$) that are equal to $1$ on $D_i$ (\resp $\wD_i$). One computes that 
\begin{equation}
\label{Pictriv}
\begin{alignedat}{2}
\Pic(X_i,D_i)&=H^1_{\Zar}(X_i,\cA_i)=H^1_{\et}(X_i,\cA_i)=H^1_{\et}(X_i,(\nu_i)_*\widetilde{\cA}_i)\\
&=H^1_{\et}(\wX_i,\widetilde{\cA}_i)=H^1_{\Zar}(\wX_i,\widetilde{\cA}_i)=\Pic(\wX_i,\wD_i),
\end{alignedat}
\end{equation}
where the first and sixth equalities are the cocycle descriptions of Picard groups, the second and fifth equality stem from \'etale descent, the third equality holds since~$\cA_i\isoto(\nu_i)_*\widetilde{\cA}_i$ by (\ref{diagconductor}), and the fourth follows from \cite[Proposition~\href{https://stacks.math.columbia.edu/tag/03QP}{03QP}]{SP} by the Leray spectral sequence for $\nu_i$.

Fix a trivialization of $\cL$ on $X'_{i-1}$ (there exists one by our choice of $i$). Extend the induced trivialization of $\cL|_{X_i}$ on $Z_i$ to a trivialization $\varphi_i$ of $\cL|_{X_i}$ on $D_i$. Define~$(\widetilde{\cL}_i,\widetilde{\varphi}_i):=(\nu_i)^*(\cL|_{X_i},\varphi_i)\in \Pic(\wX_i,\wD_i)$.
By Lemma \ref{lemCFT}, there exists a divisor~$\wE_i$ on $\wX_i\setminus \wD_i$ whose support only consists of closed points whose degree over~$\F_q$ is a multiple of $d$, such that~$(\widetilde{\cL}_i,\widetilde{\varphi}_i)\simeq\cO_{\wX_i}(\wE_i)$ in~$\Pic(\wX_i,\wD_i)$. Consider the divisor~$E_i:=(\nu_i)_*\wE_i$ on $X_i$. As $\nu_i$ is an isomorphism above $X_i\setminus D_i$ and $\Xi\cap X_i\subset D_i$, the support of $E_i$ consists of smooth closed points of $X_i\setminus (\Xi\cap X_i)$ whose degree over~$\F_q$ is a multiple of $d$.  Using~(\ref{Pictriv}), we see that $(\cL|_{X_i},\varphi_i)\simeq\cO_{X_i}(E_i)$ in $\Pic(X_i,D_i)$.

The last equality implies that $\cL(-E_i)\in\Pic(X)^d$ admits trivializations in restriction to $X'_{i-1}$ and $X_i$ that are compatible on $Z_i=X_i\cap X'_{i-1}$ (as $Z_i\subset D_i)$. This shows that the line bundle $\cL(-E_i)|_{X'_i}$ is trivial. To conclude, one applies the induction hypothesis to the line bundle $\cL(-E_i)$.
\end{proof}

\subsection{Line bundles on curves over \texorpdfstring{$p$}{p}-adic fields}
\label{pareven}

The following proposition is an analogue of Proposition \ref{propevenF} over $p$-adic fields.

\begin{prop}
\label{propeven}
Let $k$ be a $p$-adic field with residue field $\kappa$. Let $\pi:Y\to \Spec(\cO_k)$ be a flat projective morphism of relative dimension $1$ with $Y$ regular. 
Let $(X_i)_{i\in I}$ be the reduced irreducible components, with multiplicities $(m_i)_{i\in I}$, of the special fiber~$X$ of $\pi$.
Let $Z\subset Y$ be an effective divisor. Fix~$d\geq 1$ and~$\cN\in\Pic(Y)$. Let~$\delta_i\in \Z$ be the degree of~$\cN|_{X_i}$ over $\kappa$. Assume that $d\mid m_i\delta_i$ for all $i\in I$.
Then~$\cN\simeq\cO_Y(D)$ for some divisor $D$ on $Y$ whose support is a union of integral divisors that are flat of degree divisible by $d$ over~$\Spec(\cO_k)$, and that are not included in $Z$.
\end{prop}
%No converse, e.g. if X has two components of multiplicity 1 meeting transversally at one point and D meets both of them transversally at this point. Then d=2, m_i=1 and \delta_i=1. 

\begin{proof}
Let $X^{\red}$ be the reduction of $X$. Define $\Xi:=Z\cap X^{\red}$.

Fix $i\in I$. Let $U_i\subset X_i$ be the dense open subset consisting of those smooth points that do not belong to $\Xi$ or to any of the $X_j$ with $j\neq i$. Let $\kappa_i$ be the algebraic closure of $\kappa$ in $U_i$, so $U_i$ is a geometrically integral variety over $\kappa_i$.
By definition of $\delta_i$, the variety $U_i$ carries a divisor of degree $\delta_i$ over $\kappa$.
 By Lemma~\ref{lemLW} applied to the variety $U_i$ over $\kappa_i$, one can therefore find a closed point $x_i\in U_i$ whose degree~$[\kappa(x_i):\kappa]$ over $\kappa$ is congruent to $\delta_i$ modulo $d$.
 Choose an integral divisor~$D_i\subset Y$ meeting~$X^{\red}$ transversally at $x_i$ (and hence only at $x_i$ by henselianity of~$\cO_k$, see~\cite[Lemma~\href{https://stacks.math.columbia.edu/tag/04GH}{04GH}\,(1)]{SP}). The degree of $D_i$ over $\Spec(\cO_k)$ is equal to the intersection number of $D_i$ with $X$. As this number is equal to~${m_i\cdot[\kappa(x_i):\kappa]}$, it is divisible by~$d$ by hypothesis.
After replacing~$\cN$ with~$\cN(-\sum_{i\in I}D_i)$, we may therefore assume that~$\cN|_{X_i}$ has degree over $\kappa$ divisible by $d$ for all~$i\in I$.

For each $i\in I$, use Lemma~\ref{lemLW} to find a closed point $y_i\in U_i$
whose degree over~$\kappa$ is divisible by $d$. 
By Proposition \ref{propevenF}, there exists a divisor $E$ on $X^{\red}$, whose support consists of smooth closed points of $X^{\red}\setminus \Xi$ with degree over $\kappa$ divisible by~$d$, with the property that~${\cN|_{X^{\red}}\simeq\cO_{X^{\red}}(E)}$. Write~$E=F-G$ as the difference of two effective divisors~$F$ and $G$ with disjoint supports. 
Set $G:=\sum_{j\in J} n_jz_j$. 
For~$i\in I$, choose an integral divisor $D'_i\subset Y$ intersecting~$X^{\red}$ transversally at~$y_i$ (hence only at $y_i$ by henselianity of $\cO_k$, see \cite[Lemma~\href{https://stacks.math.columbia.edu/tag/04GH}{04GH}\,(1)]{SP}). Similarly, for~$j\in J$, choose an integral divisor~$D''_j\subset Y$ intersecting $X^{\red}$ transversally at $z_j$, and only at~$z_j$.

As $\cO_Y(\sum_{i\in I}D'_i)|_{X^{\red}}=\cO_{X^{\red}}(\sum_{i\in I}y_i)$ is ample, so is $\cO_Y(\sum_{i\in I}D'_i)|_{X}$ by \cite[Lemma~\href{https://stacks.math.columbia.edu/tag/09MW}{09MW}]{SP}, and hence so is 
$\cO_Y(\sum_{i\in I}D'_i)$ by \cite[Lemma~\href{https://stacks.math.columbia.edu/tag/0D2S}{0D2S}]{SP}. 
We deduce that the group $H^1(Y,\cN\otimes\cI_{X^{\red}}(N\sum_{i\in I}D'_i+\sum_{j\in J}n_jD''_j))$ (where $\cI_{X^{\red}}\subset\cO_Y$ is the ideal sheaf of $X^{\red}$ in $Y$) vanishes for~$N\geq 0$ large enough, and hence that the restriction map
$$H^0(Y,\cN(N\sum_{i\in I}D'_i+\sum_{j\in J}n_jD''_j))\to H^0(X^{\red},\cN|_{X^{\red}}(N\sum_{i\in I}y_i+\sum_{j\in J}n_jz_j))$$
is surjective. Lift an equation of $F+N\sum_{i\in I}y_i$ in $X^{\red}$ (which is an element of $H^0(X^{\red},\cN|_{X^{\red}}(N\sum_{i\in I}y_i+\sum_{j\in J}n_jz_j))$) to~$H^0(Y,\cN(N\sum_{i\in I}D'_i+\sum_{j\in J}n_jD''_j))$. By construction, the zero locus of this lift is a divisor $D'''\subset Y$ whose restriction to~$X^{\red}$ is equal to~$F+N\sum_{i\in I}y_i$. To conclude, set $$D:=D'''-N\sum_{i\in I}D_i'-\sum_{j\in J}n_jD_j''.\eqno\qedhere$$
\end{proof}

\subsection{Values of sections of line bundles over \texorpdfstring{$p$}{p}-adic curves}
\label{parvalues}

The next proposition controls the closed points of a $p$-adic curve at which a given section of the square of a line bundle is a nonzero square.

\begin{prop}
\label{propsquare}
Let $C$ be a connected smooth projective curve over a $p$-adic field~$k$. 
Let~$\cL$ be a line bundle on $C$. Choose ${\sigma\in H^0(C,\cL^{\otimes 2})}$ nonzero with reduced zero locus. 
Then there exists a divisor $\Delta$ on $C$ with~$\cL\simeq\cO_C(\Delta)$ such that
the section~$\sigma$ is nonzero and not a square at $x$,
for all those closed points $x$ in the support of $\Delta$ that have odd degree over $k$.
\end{prop}

\begin{proof}
Let $\kappa$ he the residue field of $\cO_k$.
Let~$\pi:Y\to \Spec(\cO_k)$ be a flat projective morphism with $Y$ regular whose generic fiber is isomorphic to $C$.
Extend $\cL$ to a line bundle $\cN$ on $Y$ in such a way that $\sigma$ extends to a section $\tau\in H^0(Y,\cN^{\otimes 2})$.
Let~$X$ be the special fiber of $\pi$. 
After replacing $Y$ with a modification, we may assume that $X\cup\{\tau=0\}$ is a simple normal crossings divisor in $Y$.
Replacing~$\cN$ with $\cN(-E)$ for some well-chosen divisor $E$ on~$Y$ supported on $X$, and dividing~$\tau$ by the square of an equation of~$E$, we may assume that $\tau$ has reduced zero locus.

Let $(X_i)_{i\in I}$ be the reduced irreducible components of $X$. We view them as varieties over~$\kappa$. Let~$(m_i)_{i\in I}$ be their multiplicities in $X$.
 Define 
 $$J:=\{i\in I\mid \cN|_{X_i}\textrm{ has odd degree over }\kappa\textrm{ and }m_i\textrm{ is odd}\}.$$ 
 For $j\in J$, let $\kappa_j$ be the algebraic closure of $\kappa$ in $X_j$.
 As $X_j$ carries a line bundle of odd degree over $\kappa$ (by definition of $J$), we see that $\kappa_j$ is an odd degree extension of~$\kappa$. Moreover, since $X_j$ is smooth, it is geometrically integral over $\kappa_j$. 
 We claim that for all $j\in J$, there exists an integral divisor $D_j\subset Y$ that is flat over~$\Spec(\cO_k)$, such that~$\cO_Y(D_j)|_{X_j}$ has odd degree over $\kappa$, 
 such that $D_j\cap X_i=\varnothing$ if~$i\neq j$, and such that~$\sigma$ is nonzero and not a square at the generic point of~$D_j$.

To prove the claim, we first assume that $\tau$ vanishes identically on $X_j$. 
Choose a closed point $x\in X_j$ that does not belong to any other irreducible component of~$X\cup\{\tau=0\}$ and that has odd degree over $\kappa_j$, hence over $\kappa$ (such a point exists by the Lang--Weil estimates \cite[Corollary~3]{LangWeil}).
Let $D_j\subset Y$ be an integral divisor that meets~$X_j$ transversally at~$x$. By henselianity of $\cO_k$, the divisor $D_j$ meets $X$ only at the point~$x$ (use \cite[Lemma~\href{https://stacks.math.columbia.edu/tag/04GH}{04GH}\,(1)]{SP}).
It follows that~$\cO_Y(D_j)|_{X_j}$ has odd degree over $\kappa$. In addition, since~$\tau$ vanishes at order exactly one along $X_j$, its restriction to~$D_j$ vanishes at order exactly one along $x$. Consequently, the section~$\tau$ cannot be a square at the generic point of $D_j$.

Assume now that $\tau$ does not vanish identically on $X_j$. The zero locus of $\tau$ on~$X_j$ is reduced (because $X\cup\{\tau=0\}$ is simple normal crossings and~${\{\tau=0\}}$ is reduced) and nonempty (otherwise $\tau$ would trivialize $(\cN^{\otimes 2})|_{X_j}$ and $\cN|_{X_j}$ could not have odd degree over $\kappa$). The double cover $\mu:\whX_j\to X_j$ with equation~${\{z^2=a\tau\}}$ (where~$a\in\kappa^*$ represents the nontrivial square class) is therefore geometrically integral over~$\kappa_j$. 
Let $U_j\subset \whX_j$ be the dense open subset consisting of those points whose images by $\mu$ do not belong to any irreducible component of~$X\cup\{\tau=0\}$ other than $X_j$.
By the Lang--Weil estimates \cite[Corollary 3]{LangWeil},
one can find a closed point~$y\in U_j$ that has odd degree over~$\kappa_j$, hence also over~$\kappa$.
By construction, the section $\tau$ is nonzero and not a square at $x:=\mu(y)$.
As above, any integral divisor~$D_j\subset Y$ that meets~$X_j$ transversally at $x$ is such that $\cO_Y(D_j)|_{X_j}$ has odd degree over $\kappa$. In addition, the section~$\tau$ is nonzero and not a square at the generic point of~$D_j$ because such is the case at $x$ (this can be deduced from \cite[Lemma 19.5]{EKM}).
This completes the proof of the claim.

Note that $\cN(-\sum_{j\in J}D_j)|_{X_i}$ has even degree over $\kappa$ for all $i\in I$ such that~$m_i$ is odd. By Proposition \ref{propeven} applied with $d=2$ and $Z:=\{\tau=0\}$, one can write 
$$\cN(-\sum_{j\in J}D_j)\simeq\cO_Y(D),$$
 where all the irreducible components of the support of $D$ are flat of even degree over $\Spec(\cO_k)$, and none of them are included in $\{\tau=0\}$. To conclude, define $\Delta$ to be the restriction of $D+\sum_{j\in J}D_j$ to $C$.
\end{proof}

\begin{rem}
In the setting of Proposition \ref{propsquare}, it is in general not true that all line bundles~$\cL'\in \Pic(C)$ are of the form $\cO_C(\Delta)$ with $\Delta$ as above
(\ie such that for all odd degree closed points $x$
in the support of $\Delta$, the section $\sigma$ is nonzero and not a square at~$x$).
An example is provided by $C=\P^1_k$, with $\cL=\cO_{\P^1_k}$ and~$\sigma=1\in H^0(\P^1_k,\cO_{\P^1_k})$. Then~$\cL'=\cO_{\P^1_k}(1)$ is not of the required form.
\end{rem}

The following lemma is well-known.

\begin{lem}
\label{lemrg3}
Let $k$ be a $p$-adic field. 
Let $q$ be a nondegenerate quadratic form of rank $\geq 3$ over $k$. 
Then $q$ is isotropic over~all finite extensions $l/k$ of even degree.
\end{lem}

\begin{proof}
We may assume that $q$ has rank $3$.
Let $d\in k^*$ be (a representative of) the determinant of $q$. Then~${\langle 1\rangle\perp dq}$ is the norm form of a quaternion algebra over~$k$ (see \cite[III.2]{Lam}), which splits over $l$ (see \cite[XIII.3,~Proposition~7]{Corpslocaux}). So~${\langle 1\rangle\perp dq}$ is hyperbolic over $l$ (see \cite[III.2.7]{Lam}), and $q$ is isotropic over $l$.
\end{proof}

We may now prove the main result of this section.

\begin{prop}
\label{propmiracle}
Let $C$ be a connected smooth projective curve over a $p$-adic field~$k$. Let~$\cL$ be a line bundle on $C$. Choose~$\sigma\in H^0(C,\cL^{\otimes 2})$ nonzero with reduced zero locus.
Let $q$ be a nondegenerate quadratic form of rank $\geq 3$ over $k$. 
There exists a divisor~$\Delta$ on $C$ with~$\cL\simeq\cO_C(\Delta)$ such that $\sigma$ is nonzero and represented by~$q$ at all points of the support of~$\Delta$.
\end{prop}

\begin{proof}
We may assume that $q$ has rank $3$. As isotropic forms are universal (see \cite[I.3.4]{Lam}), we may also assume that $q$ is anisotropic.

 By \cite[VI.2.15\,(2)]{Lam}, the form $q$ is not universal. Fix $a\in k^*$ not represented by $q$. 
After replacing $\sigma$ and $q$ with $a\sigma$ and $aq$, we may assume that~$a=1$.
Then~$1$ is not represented by $q$ on any odd degree extension of $k$, by Springer's theorem \cite[VII.2.9]{Lam}. We deduce from \cite[VI.2.15\,(2)]{Lam} that 
over any odd degree extension of $k$, an element is represented by $q$ if and only if it is not a square. 
In addition, the form $q$ is isotropic, hence universal, over all even degree extensions of~$k$, by Lemma \ref{lemrg3}.
The proposition now follows from Proposition~\ref{propsquare}.
\end{proof}

\section{Curves over number fields}
\label{secglobal}

We finally present the proofs of our main results.

\subsection{The representation theorem}
\label{parPourchet}

We start with Theorem~\ref{main2}. 

\begin{thm}
\label{main2bis}
Let $C$ be a geometrically connected smooth projective curve over a number field~$k$. Let $q$ be a nondegenerate quadratic form of rank $r\geq 5$ over $k$. Fix~$f\in k(C)^*$. Write $\ddiv(f)=E-2D$ with $E$ a reduced effective divisor.

 Then~$q$ represents $f$ in $k(C)$ if and only if there exists $\cM\in\Pic(C)$
such that:
\begin{enumerate}[label=(\roman*)] 
\item
\label{itext}
 if $v$ is a real place of $k$ and $q_v$ is positive definite (\resp negative definite), then $f$ is nonnegative  (\resp nonpositive) at $v$ and $\cl_v(\cM)=0$;
 \item
 \label{iitext}
if $v$ is a place of $k$ with $q_v=\tilde{q}_v\perp\langle 1,-1\rangle$ for some quadratic form~$\tq_v$ over~$k_v$, then
there exist a line bundle $\cP\in\Pic(C_v)$ and a divisor $\Delta$ on~$C_v$ with ${\cM\otimes\cP^{\otimes 2}\simeq\cO_{C_v}(\Delta-D)}$, such that~$f$ is invertible at $x$ and $f(x)$ is represented by $\tilde{q}_v$ for all closed points $x$ in the support of~$\Delta$.
\end{enumerate}
\end{thm}

\begin{proof}
Set $\cL:=\cO_C(D)$ and let $\tau$ be a rational section of $\cL$ with $\ddiv(\tau)=D$. Define~$\sigma:=f\tau^2$, so that $\sigma\in H^0(C,\cL^{\otimes 2})$ is such that $\ddiv(\sigma)=E$.

Write $q=\langle a_1,\dots, a_r\rangle$ for some $a_i\in k^*$. 

\begin{Step}
Conditions \ref{itext} and \ref{iitext} are necessary.
\end{Step}
Assume that $f=\sum_{i=1}^ra_ig_i^2$ for some~$g_i\in k(C)$. Then $\sigma=\sum_{i=1}^ra_i(g_i\tau)^2$.
Let~$F$ be the smallest effective divisor on~$C$ such that $F+\ddiv(g_i\tau)$ is effective for all $1\leq i\leq r$ with $g_i\neq 0$.
Define ${\cM:=\cO(F)}$. Let $\alpha\in H^0(C,\cM)$ be a nonzero section such that~$\ddiv(\alpha)=F$. Set $\beta_i:=g_i\alpha\tau\in H^0(C,\cL\otimes\cM)$. Then
\begin{equation}
\sigma\alpha^2=\sum_{i=1}^ra_i\beta_i^2.
\end{equation}

As $\sigma$ only has simple zeros, the $(g_i\tau)_{1\leq i\leq r}$ have no common zero. It follows from our choice of $\alpha$ that the $(\beta_i)_{1\leq i\leq r}$ have no common zero either. Condition \ref{itext} now 
follows from Proposition \ref{condition1prop} (applied to $\sigma$ if $q_v$ is positive definite and to $-\sigma$ if $q_v$ is negative definite) and condition~\ref{iitext} from Proposition \ref{condition2prop}.

\begin{Step}
Conditions \ref{itext} and \ref{iitext} are sufficient.
\end{Step}

Assume that $\cM\in\Pic(C)$ satisfies conditions \ref{itext} and \ref{iitext}. Fix $\cA\in \Pic(C)$ ample.

We may freely multiply $q=\langle a_1,\dots, a_r\rangle$ and $f$ by the same element of $k^*$. An appropriate choice of scalar (and a coordinate change) allow us to assume that $(a_1,a_2,a_3,a_4)=(1,a,b,ab)$ for some $a,b\in k^*$ (\ie that 
$q=\llangle a,b\rrangle\perp\langle a_5,\cdots,a_{r}\rangle$). It suffices to verify this claim when $q$ has rank $r=5$, in which case it is proven in \cite[Proposition 8]{Pourchet}.

Let $V$ be the finite set of places of $k$ at which the Pfister form $\llangle a,b\rrangle$ is anisotropic (\ie at which the quaternion algebra $(a,b)$ does not split, see \cite[III.2.7]{Lam}).

Fix $v\in V$. Unless $v$ is real and $q_v$ is definite, the quadratic form $q_v$ is isotropic (see \cite[VI.2.12]{Lam}). It follows from Proposition \ref{condition1prop} (if $v$ is a real and~$q_v$ is definite) or from Proposition \ref{condition2prop} (otherwise) that for all $l$ large enough, there exist sections~$\alpha_v\in H^0(C_v,\cM\otimes\cA^{\otimes 2l})$ and~$\beta_{i,v}\in H^0(C_v,\cL\otimes\cM\otimes\cA^{\otimes 2l})$ such that
\begin{equation}
\sigma\alpha_v^2=\sum_{i=1}^ra_i\beta_{i,v}^2
\end{equation}
with $\alpha_v\neq 0$. In addition, the quoted propositions allow us to ensure that no two of the $(\beta_{i,v})$ have a common real zero (if $v$ is real and $q_v$ is definite) or a common zero (otherwise).
Finally, as~$V$ is finite, one can choose $l$ to be independent of $v\in V$.

Use the Artin--Whaples approximation theorem \cite[Theorem 1]{AW},
%cf [Roquette, History of valuation theory, Part I, 4.2.1] for historical comments.
to find sections $\alpha\in H^0(C,\cM\otimes\cA^{\otimes 2l})$ and~$\beta_{i}\in H^0(C,\cL\otimes\cM\otimes\cA^{\otimes 2l})$ that are close to the $\alpha_v$ and the $\beta_{i,v}$ for the topology induced by $v$, for all $v\in V$.

We claim that~$f-\sum_{i=5}^ra_i\big(\frac{\beta_i}{\alpha\tau}\big)^2=\frac{1}{\alpha^2\tau^2}(\sigma\alpha^2-\sum_{i=5}^ra_i\beta_i^2)$ is represented by~$\llangle a,b\rrangle$ in $k_v(C_v)$ for all places $v$ of $k$. 
\begin{enumerate}
[label=(\alph*)] 
\item
If $v\notin V$, then $\llangle a,b\rrangle$ is isotropic, hence universal (see \cite[I.3.4]{Lam}), and the claim is trivial.
\item
 If $v\in V$ is real then the form $\llangle a,b\rrangle$ is definite, hence positive definite, at the place $v$. As the $(\beta_{i,v})_{1\leq i\leq 4}$ have no common real zero, we deduce that the section $\sigma\alpha_v^2-\sum_{i=5}^ra_i\beta_{i,v}^2=\sum_{i=1}^4a_i\beta_{i,v}^2\in H^0(C_v,(\cL\otimes\cM\otimes\cA^{\otimes 2l})^{\otimes 2})$ is positive on $C_v(k_v)$. If $\alpha$ and the $\beta_i$ have been chosen close enough to~$\alpha_v$ and the $\beta_{i,v}$, then so is the section~$\sigma\alpha^2-\sum_{i=5}^ra_i\beta_{i}^2$. It follows that~$f-\sum_{i=5}^ra_i\big(\frac{\beta_i}{\alpha\tau}\big)^2$ is nonnegative.
By a theorem of Witt (see \cite[I~p.\,4]{Witt}), we deduce that $f-\sum_{i=5}^ra_i\big(\frac{\beta_i}{\alpha\tau}\big)^2$ is a sum of two squares in $k_v(C_v)$, hence that it is represented by $\llangle a,b\rrangle$ in this field.
\item
Suppose that~$v\in V$ is $p$-adic for some prime number $p$. Let $x\in C_v$ be a zero of $\sigma\alpha_v^2-\sum_{i=5}^ra_i\beta_{i,v}^2=\sum_{i=1}^4a_i\beta_{i,v}^2\in H^0(C_v,(\cL\otimes\cM\otimes\cA^{\otimes 2l})^{\otimes 2})$. As the~$(\beta_{i,v})_{1\leq i\leq 4}$ do not all vanish at $x$, the form $\llangle a,b\rrangle=\langle a_1,a_2,a_3,a_4\rangle$ is isotropic in the residue field of $x$. It therefore follows from Proposition~\ref{openness} that, if $\alpha$ and the $\beta_i$ have been chosen close enough to~$\alpha_v$ and the $\beta_{i,v}$, the section~$\sigma\alpha^2-\sum_{i=5}^ra_i\beta_{i}^2$ is represented by $\llangle a,b\rrangle$ at the generic point of~$C_v$. It follows that $f-\sum_{i=5}^ra_i\big(\frac{\beta_i}{\alpha\tau}\big)^2$ is represented by $\llangle a,b\rrangle$ in $k_v(C_v)$.
\end{enumerate}
We deduce from Proposition \ref{localglobal} that $f-\sum_{i=5}^ra_i\big(\frac{\beta_i}{\alpha\tau}\big)^2$ is represented by $\llangle a,b\rrangle$ in~$k(C)$, hence that $f$ is represented by $q$ in $k(C)$.
 \end{proof}
 
 Let us list a few cases in which condition \ref{iitext} can be verified.
 
\begin{rems}
\label{rem41}
(i)
Condition \ref{iitext} in Theorem \ref{main2bis} is always satisfied, for any choice of~${\cM\in\Pic(C)}$, if $\tq_v$ is isotropic (as isotropic forms are universal \cite[I.3.4]{Lam}). 
 
(ii)
Assume that $v$ is real and that $\tq_v$ is positive definite (\resp negative definite). Then condition  \ref{iitext} in Theorem \ref{main2bis} holds for a given $\cM\in \Pic(C)$ if and only if, for any connected component~$\Gamma$ of $C_v(\R)$ on which~$f$ is nonpositive (\resp nonnegative), the class~$\cl_v(\cM(D))|_{\Gamma}\in H^1(\Gamma,\Z/2)=\Z/2$ vanishes (see Lemma \ref{lemhyp} below).
 
(iii)
If $v$ is a $p$-adic place, then condition \ref{iitext} in Theorem \ref{main2bis} is always satisfied for~$\cM=\cO_{C}$, by Proposition \ref{propmiracle} (one can take $\cP=\cO_{C_v}$).
\end{rems}
 
 \begin{lem}
 \label{lemhyp}
 Let $C$ be a connected smooth projective curve over $\R$. Fix $f\in\R(C)^*$.
 Write $\ddiv(f)=E-2D$ with $E$ a reduced effective divisor. For all $\cM\in \Pic(C)$, the following assertions are equivalent.
 \begin{enumerate}
[label=(\roman*)] 
\item
\label{iR}
There exist $\cP\in\Pic(C)$ and a divisor $\Delta$ on $C$ with ${\cM\otimes\cP^{\otimes 2}\simeq\cO_{C}(\Delta-D)}$ such that $f$ is invertible and a square at $x$, for all $x$ in the support of $\Delta$.
\item
\label{iiR}
For any connected component~$\Gamma$ of $C(\R)$ on which~$f$ is nonpositive, the Borel--Haefliger class~$\cl_{\R}(\cM(D))|_{\Gamma}\in H^1(\Gamma,\Z/2)=\Z/2$ vanishes.
\end{enumerate}
 \end{lem}
 
 \begin{proof}
 Assume that \ref{iR} holds. If $\cl_{\R}(\cM(D))|_{\Gamma}\neq 0$, then $\cl_{\R}(\Delta)|_{\Gamma}\neq 0$, and the support of $\Delta$ must meet $\Gamma$ at some point. At such a point, the rational function $f$ is invertible and a square, hence positive. This proves \ref{iiR}.

 Conversely, suppose that \ref{iiR} holds. Then there exists a divisor~$\Delta$ on $C$ such that~$f$ is invertible and positive at all real points of the support of $\Delta$, and such that~${\cl_{\R}(\cM(D-\Delta))=0}$. By a theorem of Witt (see \cite[III~p.\,4]{Witt}), one can modify $\Delta$ by a divisor whose support has no real points to ensure that moreover~$\cM(D-\Delta)\simeq\cO_{C}$. Write $\Delta=\sum_i n_ix_i+\sum_jm_jy_j$, where the $x_i$ (\resp the~$y_j$) are closed points of $C$ with complex (\resp real) residue fields. Choose $N\geq 0$ such that $\cO_C(Nx_i)$ is very ample for all $i$. Let $\sigma_i\in H^0(C,\cO_C(x_i))$ be an equation of~$x_i$. Let $\tau_i\in H^0(C,\cO_C(Nx_i))$ be a small perturbation of $\sigma_i^N$, chosen not to vanish on any real point of $C$ or on any point of $C$ at which $f$ is not invertible. After replacing~$\Delta$ with~$\Delta+\ddiv(\prod_i\frac{\tau_i}{\sigma_i^N})$, we may assume that $f$ is invertible at all points of the support of~$\Delta$. This proves~\ref{iR}.
 \end{proof}

 \subsection{Consequences of Theorem \ref{main2bis}}
 \label{parcoro}
 
 We now derive more concrete corollaries of Theorem \ref{main2bis} under various additional assumptions (on the function $f$, on the quadratic form $q$, or on the curve $C$).

\begin{cor}
\label{coros}
Let $C$ be a geometrically connected smooth projective curve over a number field~$k$. Let $q$ be a nondegenerate quadratic form of rank $r\geq 5$ over~$k$. Fix~${f\in k(C)^*}$. Write $\ddiv(f)=E-2D$ with $E$ a reduced effective divisor. Assume that at least one of the following assertions holds.
\begin{enumerate}[label=(\alph*)] 
\item
\label{atext}
$\cl_v(D)=0$ for all real places $v$ of $k$ at which $q_v$ is definite.
\item
\label{btext}
$\cl_v(D)=0$ for all real places $v$ of $k$ at which $q_v$ 
has hyperbolic~signature.
\item 
\label{ctext}
 The form $q$ does not have hyperbolic signature at any real place of~$k$. 
 \item
  \label{dtext}
 The curve $C$ is a nonsplit conic over $k$.
 \item
 \label{etext}
One can write $\ddiv(f)=4F+G$ in such a way that the residue fields of all points in the support of~$G$ have no real embeddings.
\end{enumerate}
Then~$q$ represents $f$ in $k(C)$ if and only if $f$ is nonnegative (\resp non\-positive) at all real places of $k$ at which $q$ is positive definite (\resp negative definite).
\end{cor}

\begin{proof}
In case \ref{atext}, we apply Theorem \ref{main2bis} with $\cM=\cO_C(-D)$. Condition \ref{itext} is satisfied by our hypotheses on the real places at which $q$ is definite. Condition \ref{ii} is automatically satisfied (with $\cP=\cO_{C_v}$ and $\Delta=0$) by our choice of $\cM$.

In case \ref{btext}, we apply Theorem~\ref{main2bis} with $\cM=\cO_C$. 
Condition \ref{itext} is satisfied by our hypothesis on definite real places $v$ of $k$, as $\cl_v(\cO_C)=0$. Condition \ref{iitext} holds at a real place by Remarks~\ref{rem41} (i) and (ii) and 
at a $p$-adic place by Remark \ref{rem41} (iii).

Case~\ref{ctext} is a particular case of \ref{btext}.

Case \ref{dtext} follows from either \ref{atext} or \ref{btext}. Indeed, 
let $v$ be a real place of~$k$. The set of real points $C_v(\R)$ is empty or connected (because $C$ is a conic).  In addition, the group $\Pic(C)$ is generated by a degree $2$ line bundle (as $C$ is a nonsplit conic). It follows from these two facts that $\cl_v$ vanishes identically on $\Pic(C)$.

Finally, case \ref{etext} also follows from either \ref{atext} or \ref{btext}, as it implies that $\cl_v(D)=0$ for all real places $v$ of $k$.
\end{proof}
 
 \begin{rems}
 \label{remmain2}
 (i) 
 Condition \ref{etext} in Corollary \ref{coros} is exactly the one appearing in Pop's \cite[Theorem 2.5\,(1)]{Pop} (in the particular case where $q=\langle 1,1,1,1,1\rangle$). One can therefore think of Corollary~\ref{coros}\,\ref{etext} as generalizing this result~of~Pop.
 
 (ii)
If $q$ is not definite at any real place of $k$, it follows from Corollary \ref{coros}\,\ref{atext} that any~$f\in k(C)$ is represented by $q$. However, this statement is trivial because this assumption implies that $q$ is isotropic (see \cite[VI.3.5]{Lam}) hence universal over any extension of $k$ (see \cite[I.3.4]{Lam}).
 
 (iii)
 The argument in (ii) also shows that Theorem \ref{main2bis} remains true, and is trivial, over~global fields of positive characteristic (in characteristic $2$, see \cite[Proposition~3.1]{Pollak}).
 \end{rems}
 
In the case $C=\P^1_k$, we recover the results of Pourchet \cite[Corollaire 1 p.\,98]{Pourchet}.

\begin{cor}
\label{coroPourchet}
Let $k$ be a number field. Let $q$ be a nondegenerate quadratic form of rank $r\geq 5$ over $k$. Let $f\in k[t]$ be a separable polynomial. Then $q$ represents $f$ over $k[t]$ (equivalently, over $k(t)$) if and only if one of the following assertions hold.
\begin{enumerate}[label=(\roman*)] 
\item
\label{Pi}
For all real places $v$ of $k$, the form $q_v$ is not definite.
\item
\label{Pii}
For all real places $v$ of $k$ such that $q_v$ is positive definite (\resp negative definite), the function $f$ is nonnegative (\resp nonpositive) at $v$, and either
\begin{enumerate}[label=(\alph*)] 
\item
\label{Pa}
the degree of $f$ is divisible by $4$; or
\item
\label{Pb}
if $v$ is a real place of $k$ and $q_v$ has signature~$(r-1,1)$ (\resp $(1,r-1)$), then $f$ is not nonpositive (\resp not nonnegative) at $v$.
\end{enumerate}
\end{enumerate}
\end{cor}

\begin{proof}
That $q$ represents $f$ over $k[t]$ if and only if it represents it over $k(t)$ follows from the Cassels--Pfister theorem (see \cite[IX.1.3]{Lam}). To analyze when $q$ represents $f$ over $k(t)$, we apply Theorem \ref{main2bis} with $C=\P^1_k$. 

 If $q$ is indefinite at all real places of $k$, then $q$ represents $f$ in $k(t)$ (see Remark~\ref{remmain2}~(ii)). 
 Assume otherwise. Then the necessary condition that $f$ be nonnegative (\resp non\-positive) at positive definite (\resp negative definite) real places can only be satisfied if $f$ has even degree (otherwise it changes sign). Writing~$\deg(f)=2n$, one has $\cO_C(D)=\cO_{\P^1_k}(n)$ in the notation of Theorem \ref{main2bis}.
 
Condition \ref{itext} of Theorem \ref{main2bis} can only be satisfied with $\cM=\cO_{\P^1_k}(l)$ for~$l$ even. 
 For such a choice of $\cM$, condition \ref{iitext} of Theorem~\ref{main2bis} is satisfied at all real places if and only if either \ref{Pa} or \ref{Pb} holds (apply Remarks~\ref{rem41}\,(i) and (ii)).
 When this is the case, one can choose $\cM=\cO_{\P^1_k}$, as condition \ref{iitext} of Theorem \ref{main2bis} is then satisfied at all $p$-adic places (see Remark \ref{rem41}\,(iii)). This completes the proof.
 \end{proof}

\begin{rem}
In \cite[Corollaire 1 p.\,98]{Pourchet}, Pourchet only formulates a sufficient condition for $f$ to be represented by~$q$. He does not explicitly consider the (much easier) cases when~$q$ is indefinite at all real places of~$k$, or when $f$ has odd degree. He also restricts his statement to the (most interesting) case where $q$ has rank $5$.
\end{rem}

 \subsection{The Pythagoras number}
 \label{parPythagoras}
  
  Here is an application to the Pythagoras number.
  
 \begin{thm}
 \label{main1bis}
 Let $C$ be a connected smooth curve over a number field $k$. Then $p(k(C))\leq 5$.
 \end{thm}
 
 \begin{proof}
 We may replace $C$ with its smooth projective compactification. After replacing $k$ with its algebraic closure in $k(C)$, we may assume that $C$ is geometrically connected over $k$.
 Let $f\in k(C)$ be a sum of squares, which we may assume to be nonzero. It is nonnegative at all real places of $k$. It is therefore a sum of $5$ squares in $k(C)$ by Corollary \ref{coros}\,\ref{ctext} applied with $q:=\langle 1,1,1,1,1\rangle$.
  \end{proof}
 
 Theorem \ref{main1} is a simple consequence of Theorem \ref{main1bis}.
 
\begin{cor}
If $F$ is a field of transcendence degree $1$ over $\Q$, then $p(F)\leq 5$.
\end{cor}

\begin{proof}
Let $f\in F$ be a sum of squares. We wish to show that $f$ is a sum of $5$ squares in $F$. Replacing $F$ by a finitely generated subfield of transcendence degree~$1$ over~$\Q$ that contains $f$ and in which $f$ is a sum of squares, we may assume that~$F$ is finitely generated over $\Q$, and hence the function field of a connected smooth curve over a number field. One can then apply Theorem \ref{main1bis}.
\end{proof}

Our last goal is to explain the relation between Theorem \ref{main1bis} and Hilbert's 17th problem. The next proposition is a version of the solution to this problem given by Artin in \cite{Artin} and of its subsequent extension by Lang \cite[Theorem~9]{Lang}. We could not find it in this precise form in the literature.

\begin{prop}
\label{H17}
Let $X$ be a connected smooth variety over a field~$k$ with $2\in k^*$. 
%Artin forgets this hypothesis in Satz 1.
Fix an element $f\in\cO(X)$. The following assertions are equivalent:
\begin{enumerate}[label=(\roman*)] 
\item for all field orderings $\succeq$ of $k(X)$, one has $f\succeq 0$;
\label{i}
\item for all real closures $k\subset R$ of $k$ and all $x\in X(R)$, one has $f(x)\geq 0$;
\label{ii}
\item for all real closed extensions $k\subset R$ and all $x\in X(R)$, one has $f(x)\geq 0$;
\label{iv}
\item $f$ is a sum of squares in $k(X)$.
\label{iii}
\end{enumerate}
If $k$ is a number field, they are also equivalent to:
\begin{enumerate}[label=(\roman*)] 
\setcounter{enumi}{4}
\item for all real places $v$ of $k$ and all $x\in X_v(\R)$, one has $f(x)\geq 0$.
\label{v}
\end{enumerate}
\end{prop}

\begin{proof}
The equivalence \ref{ii}$\Leftrightarrow$\ref{iv} follows from the Tarski--Seidenberg transfer principle \cite[Proposition 5.2.3]{BCR}.
The equivalence  \ref{i}$\Leftrightarrow$\ref{iii} is a theorem of Artin \cite[Satz 1]{Artin}.
That \ref{iii} implies \ref{ii} follows from \cite[Theorem~6.1.9~(i)$\Rightarrow$(iii)]{BCR}.

Assume that \ref{i} fails. Fix a field ordering $\succeq$ of $k(X)$ with $f\prec 0$. The real closure~$S$ of $K(X)$ with respect to $\succeq$ contains the real closure $R$ of $k$ with respect to~$\succeq$. As the compositum of $R$ and $k(X)$ in $S$ is isomorphic to $R(X_R)$, the field~$R(X_R)$ admits an ordering $\succeq$ with $f\prec 0$. It follows from \cite[Theorem~6.1.9~(iii)$\Rightarrow$(i)]{BCR} that there exists $x\in X(R)$ such that $f(x)<0$. This disproves~\ref{ii}.

If $k$ is a number field, then \ref{iv}$\Rightarrow$\ref{v}$\Rightarrow$\ref{ii} because the orderings of $k$ are in bijection with its real places (see \eg \cite[Satz 10]{AS}).
\end{proof}

\begin{cor}
\label{17}
Let $C$ be a connected smooth curve over a number field $k$ and fix~$f\in k(C)$. The following assertions are equivalent.
\begin{enumerate}[label=(\roman*)] 
\item
\label{a}
 for all field orderings $\succeq$ of $k(C)$, one has $f\succeq 0$;
\item
\label{b}
 for all real places $v$ of $k$ and all $x\in C_v(\R)$  that is not a pole of~$f$, one has~$f(x)\geq 0$;
\item 
\label{c}
$f$ is a sum of squares in $k(C)$;
\item 
\label{d}
$f$ is a sum of $5$ squares in $k(C)$.
\end{enumerate}
\end{cor}

\begin{proof}
It suffices to combine Theorem \ref{main1bis} and Proposition \ref{H17}.
\end{proof}

\bibliographystyle{myamsalpha}
\bibliography{Pythagoras}

\providecommand{\bysame}{\leavevmode\hbox to3em{\hrulefill}\thinspace}
\providecommand{\MR}{\relax\ifhmode\unskip\space\fi MR }
% \MRhref is called by the amsart/book/proc definition of \MR.
\providecommand{\MRhref}[2]{%
  \href{http://www.ams.org/mathscinet-getitem?mr=#1}{#2}
}
\providecommand{\href}[2]{#2}
\begin{thebibliography}{EKM08}

\bibitem[Apo76]{Apostol}
T.~M. Apostol, \emph{Introduction to analytic number theory}, Undergraduate
  Texts in Math., Springer-Verlag, New York-Heidelberg, 1976.

\bibitem[Art27]{Artin}
E.~Artin, \emph{\"{U}ber die {Z}erlegung definiter {F}unktionen in {Q}uadrate},
  Abh. Math. Sem. Univ. Hamburg \textbf{5} (1927), no.~1, 100--115.

\bibitem[Art69]{Artinapprox}
M.~Artin, \emph{Algebraic approximation of structures over complete local
  rings}, Publ. Math. IH{\'E}S \textbf{36} (1969), 23--58.

\bibitem[AS26]{AS}
E.~{Artin} and O.~{Schreier}, \emph{{Algebraische Konstruktion reeller
  K\"orper}}, {Abh. Math. Sem. Univ. Hamburg} \textbf{5} (1926), 85--99.

\bibitem[AW45]{AW}
E.~Artin and G.~Whaples, \emph{Axiomatic characterization of fields by the
  product formula for valuations}, Bull. AMS \textbf{51} (1945), 469--492.

\bibitem[BCR98]{BCR}
J.~Bochnak, M.~Coste and M.-F. Roy, \emph{Real algebraic geometry}, Ergeb.
  Math. Grenzgeb. (3), vol.~36, Springer-Verlag, Berlin, 1998.

\bibitem[BGR84]{BGR}
S.~Bosch, U.~G{\"u}ntzer and R.~Remmert, \emph{Non-{Archimedean} analysis. {A}
  systematic approach to rigid analytic geometry}, Grundlehren Math. Wiss.,
  vol. 261, 1984.

\bibitem[BKS23]{BKS}
F.~Binda, A.~Krishna and S.~Saito, \emph{Bloch's formula for 0-cycles with
  modulus and higher-dimensional class field theory}, J. Algebr. Geom.
  \textbf{32} (2023), no.~2, 323--384.

\bibitem[BM40]{BMacL}
M.~F. Becker and S.~MacLane, \emph{The minimum number of generators for
  inseparable algebraic extensions}, Bull. AMS \textbf{46} (1940), 182--186.

\bibitem[Bor91]{Borel}
A.~Borel, \emph{Linear algebraic groups}, 2nd ed., Grad. Texts Math., vol. 126,
  1991.

\bibitem[CT86]{CTAppendix}
J.-L. Colliot-Th\'el\`ene, \emph{Appendix to a {Hasse} principle for two
  dimensional global fields by {K}. {K}ato}, J. reine angew. Math. \textbf{366}
  (1986), 181--183.

\bibitem[CTJ91]{CTJ}
J.-L. Colliot-Th{\'e}l{\`e}ne and U.~Jannsen, \emph{Sommes de carr{\'e}s dans
  les corps de fonctions}, CRAS, Paris, S{\'e}r. I \textbf{312} (1991), no.~11,
  759--762.

\bibitem[EKM08]{EKM}
R.~Elman, N.~Karpenko and A.~Merkurjev, \emph{The algebraic and geometric
  theory of quadratic forms}, Colloq. Publ., AMS, vol.~56, 2008.

\bibitem[Eul60]{Euler}
L.~Euler, \emph{Demonstratio theorematis {F}ermatiani omnem numerum sive
  integrum sive fractum esse summam quatuor pauciorumve quadratorum}, Novi
  Comment. Acad. Sci. Imp. Petropol. \textbf{5} (1760), 13--58.

\bibitem[Has23]{Hasse}
H.~Hasse, \emph{Darstellbarkeit von {Zahlen} durch quadratische {Formen} in
  einem beliebigen algebraischen {Zahlk{\"o}rper}.}, J. reine angew. Math.
  \textbf{153} (1923), 113--130.

\bibitem[Jan16]{Jannsen}
U.~Jannsen, \emph{Hasse principles for higher-dimensional fields}, Ann. Math.
  (2) \textbf{183} (2016), no.~1, 1--71.

\bibitem[JS02]{JS}
U.~Jannsen and R.~Sujatha, \emph{Levels of function fields of surfaces over
  number fields}, J. Algebra \textbf{251} (2002), no.~1, 350--357.

\bibitem[Kat86]{Kato}
K.~Kato, \emph{A {Hasse} principle for two dimensional global fields}, J. reine
  angew. Math. \textbf{366} (1986), 142--181.

\bibitem[Lam05]{Lam}
T.~Y. Lam, \emph{Introduction to quadratic forms over fields}, Graduate Studies
  in Math., vol.~67, Amer. Math. Soc., Providence, RI, 2005.

\bibitem[Lan06]{Landau}
E.~Landau, \emph{{\"U}ber die {D}arstellung definiter {F}unktionen durch
  {Q}uadrate}, Math. Ann. \textbf{62} (1906), 272--285.

\bibitem[Lan53]{Lang}
S.~Lang, \emph{The theory of real places}, Ann. of Math. (2) \textbf{57}
  (1953), 378--391.

\bibitem[LW54]{LangWeil}
S.~Lang and A.~Weil, \emph{Number of points of varieties in finite fields},
  Amer. J. Math. \textbf{76} (1954), 819--827.

\bibitem[MS83]{MS}
A.~S. Merkurjev and A.~A. Suslin, \emph{{{\(K\)}}-cohomology of
  {Severi}--{Brauer} varieties and the norm residue homomorphism}, Math. USSR
  Izv. \textbf{21} (1983), 307--340.

\bibitem[Pfi95]{Pfisterbook}
A.~Pfister, \emph{Quadratic forms with applications to algebraic geometry and
  topology}, LMS Lect. Note Ser., vol. 217, 1995.

\bibitem[Pie93]{Pieper}
H.~Pieper, \emph{On {E}uler's contributions to the four-squares theorem},
  Historia Math. \textbf{20} (1993), no.~1, 12--18.

\bibitem[Pol70]{Pollak}
B.~Pollak, \emph{Orthogonal groups over global fields of characteristic 2}, J.
  Algebra \textbf{15} (1970), 589--595.

\bibitem[Pop91]{Popnote}
F.~Pop, \emph{Summen von {Q}uadraten in arithmetischen
  {F}unktionenk{\"o}rpern},
  \href{https://www2.math.upenn.edu/~pop/Research/files-Res/dimen1.pdf}{\tt
  {https://www2.math.upenn.edu/\string~pop/Research/files-Res/dimen1.pdf}},
  1991.

\bibitem[Pop23]{Pop}
\bysame, \emph{On the {Pythagoras} number of function fields of curves over
  number fields}, Isr. J. Math. \textbf{257} (2023), no.~2, 561--574.

\bibitem[Pou71]{Pourchet}
Y.~Pourchet, \emph{Sur la repr{\'e}sentation en somme de carr{\'e}s des
  polyn{\^o}mes {\`a} une ind{\'e}termin{\'e}e sur un corps de nombres
  alg{\'e}briques}, Acta Arith. \textbf{19} (1971), 89--104.

\bibitem[Ser68]{Corpslocaux}
J.-P. Serre, \emph{Corps locaux}, Hermann, Paris, 1968.

\bibitem[Sie21]{Siegel}
C.~L. Siegel, \emph{Darstellung total positiver {Zahlen} durch {Quadrate}},
  Math. Z. \textbf{11} (1921), 246--275.

\bibitem[SP]{SP}
A.~J. de~Jong et~al., \emph{\textit{The Stacks Project}},
  \url{https://stacks.math.columbia.edu}.

\bibitem[vdW70]{vdW}
B.~L. van~der Waerden, \emph{Algebra {V}ol. 1}, Frederick Ungar Publishing Co.,
  New York, 1970.

\bibitem[Voe03]{Voevodsky}
V.~Voevodsky, \emph{Motivic cohomology with {{\(\mathbb Z/2\)}}-coefficients},
  Publ. Math. IH{\'E}S \textbf{98} (2003), 59--104.

\bibitem[Wit34]{Witt}
E.~Witt, \emph{Zerlegung reeller algebraischer {Funktionen} in {Quadrate}.
  {Schiefk{\"o}rper} {\"u}ber reellem {Funktionenk{\"o}rper}}, J. reine angew.
  Math. \textbf{171} (1934), 4--11.

\bibitem[Wit35]{WittCFT}
\bysame, \emph{Der {Existenzsatz} f{\"u}r abelsche {Funktionenk{\"o}rper}}, J.
  reine angew. Math. \textbf{173} (1935), 43--51.

\end{thebibliography}

\end{document}